\documentclass[11pt,a4paper,reqno]{amsart}
\usepackage[utf8]{inputenc}
\usepackage{amsmath, amssymb, amsthm, mathtools}
\usepackage{bbold,tabularx}
\usepackage{enumerate,float,subcaption,xcolor,multirow}
\usepackage{hyperref,tikz, tikzpagenodes}
\hypersetup{colorlinks,allcolors=black}
\usetikzlibrary{through,intersections}
\usepackage{tikz-cd}
\mathtoolsset{centercolon}
\usepackage{xurl}
\usepackage{makecell}
\usepackage{bbm}
\usepackage{enumitem}
\usepackage{comment}
\usepackage{float}

\newtheorem{theorem}{Theorem}
\newtheorem{corollary}[theorem]{Corollary}
\newtheorem{lemma}[theorem]{Lemma}
\newtheorem{example}[theorem]{Example}
\newtheorem{proposition}[theorem]{Proposition}
\theoremstyle{definition}
\newtheorem{definition}[theorem]{Definition}
\newtheorem{remark}[theorem]{Remark}

\usepackage[margin=2.8cm]{geometry}

\def\0{{\bf 0}}
\def\1{{\bf 1}}

\newcommand{\R}{{\mathbb {R}}}
\newcommand{\F}{{\mathbb {F}}}
\newcommand{\Fq}{\F_q}

\colorlet{ss}{black!20}
\colorlet{sslabel}{black!50}
\colorlet{backline}{black!10}

\newcommand{\srk}{\mathrm{srk}}
\newcommand{\rk}{\operatorname{rk}}

\makeatletter

\newcommand{\rref}[2]{\hyperref[#2]{{#1}~\ref*{#2}}}

\usepackage{xcolor}
\usepackage{eurosym}

\title{The Eigenvalue Method in Coding Theory}
\author{Aida Abiad} \thanks{\texttt{a.abiad.monge@tue.nl}, Department of Mathematics and Computer Science, Eindhoven University of Technology, The Netherlands
\thanks{Department of Mathematics and Data Science of Vrije Universiteit Brussel, Belgium}
\author{Loes Peters} \thanks{\texttt{l.peters@tue.nl}, Department of Mathematics and Computer Science, Eindhoven University of Technology, The Netherlands} 
\author{Alberto Ravagnani} \thanks{\texttt{a.ravagnani@tue.nl},  Department of Mathematics and Computer Science, Eindhoven University of Technology, The Netherlands}}

\date{}

\begin{document}

\begin{abstract}
We lay down the foundations of the Eigenvalue Method in coding theory. The method uses modern algebraic graph theory to derive upper bounds on the size of error-correcting codes for various metrics, addressing major open questions in the field. We identify the core assumptions that allow applying the Eigenvalue Method, test it for multiple well-known classes of error-correcting codes, and compare the results with the best bounds currently available. By applying the Eigenvalue Method, we obtain new bounds on the size of error-correcting codes that often improve the state of the art. Our results show that spectral graph theory techniques capture structural properties of error-correcting codes that are missed by classical coding theory approaches.\\

\medskip

\noindent \textbf{Keywords:} spectral graph theory, coding theory, error-correcting code, eigenvalue bounds, linear optimization.\\
\noindent \textbf{MSC:} 11T71, 05C50
\end{abstract}

\maketitle





\section{Introduction}


This paper is about the interplay between spectral graph theory and algebraic coding theory. Spectral graph theory focuses on describing the combinatorial properties of a graph via the eigenvalues (spectrum) of its adjacency matrix, while coding theory is the science of adding redundancy to data in such a way it becomes resistant to noise. Redundancy is added using mathematical objects called error-correcting codes, whose theory dates back to Shannon's celebrated paper ``{A mathematical theory of communication}''~\cite{shannon1948mathematical}.

There exist several classes of error-correcting codes, each of which is best suited to correct the error patterns introduced by a specific type of noisy channel. 
However, most classes of error-correcting codes can be described with the same high-level framework. The starting point is a finite ``ambient'' set $A$ endowed with a distance function $d: A \times A \to \R$, which reflects the underlying channel. The pair $(A,d)$ is called a \textit{discrete metric space}. An error-correcting code is a subset $\mathcal{C} \subseteq A$, where the distance between distinct elements is bounded from below by a given number $d^*$, measuring the correction capability of $\mathcal{C}$. There is a trade-off between having large $d^*$ and having a large cardinality: The main task in this context is to find the largest possible~$\mathcal{C}$ for a given value $d^*$. Depending on the combinatorial structure of $A$, this problem can be relatively easy~\cite{Delsarte1978BilinearTheory}, or inspire conjectures that are almost 70 years old~\cite{segre1955curve,chowdhury2016inclusion,Voloch1991, HirschfeldKorchmaros1996, BallDeBeule2012a, Ball2010}.
This paper concentrates on establishing the foundations of the \textit{Eigenvalue Method} for solving this central task. Recently, this method has been successfully applied to three distinct metrics, see ~\cite{Abiad2024EigenvalueCodes,Abiad2023EigenvalueCodes, Abiad2024EigenvalueCodesb}.

There is a natural connection between coding theory and graph theory. Let the elements of $A$ be the vertices of a graph $G$. Connect two vertices $x,y$ if their distance $d(x,y)$ is at most $d^*-1$. Then the largest cardinality of an error-correcting code with the desired correction capability is precisely the independence number of $G$. This observation has been used in various instances to obtain bounds on the size of error-correcting codes, or to revisit bounds established using different techniques; see for instance~\cite{el2007bounds, jiang2004asymptotic, krivelevich2004lower}.

Algebraic graph theory is the foundation of one of the best known methods to estimate the size of an error-correcting code, namely Delsarte's Linear Programming Bound~\cite{Delsarte1973AnTheory}. Delsarte's method makes use of an association scheme describing the properties of the space $(A,d)$ to construct a linear program, whose maximum value is an upper bound for the size of a code. Delsarte's method is widely used and applies to several classes of codes, even though it's a quite technical result that requires specific computations for each scheme at hand;  see~\cite{Delsarte1978BilinearTheory,Delsarte1975AlternatingGFq,Dukes2020OnCodes, Astola1982TheLee-codes,Sole1986TheScheme,Roy2009BoundsSubspaces,Abiad2024ACodes} among many others. Furthermore,
not all spaces $(A,d)$ come with a natural structure of an association scheme. For instance the sum-rank-metric space does not come with this natural association scheme, but an alternative scheme was recently derived~\cite{Abiad2024ACodes}.
In sharp contrast with Delsarte's approach, the method proposed in this paper does not rely on association schemes and it only requires computing the spectrum of a graph. Even when Delsarte's method can be used, the approach proposed in this paper is easier to apply and provides competitive bounds. Furthermore, for small minimum distances, the Eigenvalue Method provides closed formulas and therefore the optimal polynomials, while for Delsarte's approach this is not known for most metrics. Such closed formulas for the bounds from the Eigenvalue Method can then be used to show non-existence and characterization results for several metrics, as it was done for instance for the sum-rank metric \cite{Abiad2023EigenvalueCodes} and for the Lee metric \cite{Abiad2024EigenvalueCodesb}.

The Eigenvalue Method, which is the centerpiece of this paper, stems from the observation that, for several ambient spaces $A$ relevant for coding theory, the graph $G$ defined above is the $(d^*-1)$-th power graph of a simpler graph $G'$. When this happens, the independence number of $G$ is the $(d^*-1)$-independence number of $G'$. The graph $G'$ is defined as follows: Instead of connecting $x$ and $y$ if $d(x,y) \le d^*-1$, we connect them if $d(x,y)=1$.

Interestingly, several ambient spaces
$A$ that arise in coding theory naturally have the regularity properties that are needed to write $G$ as the power graph of a graph $G'$. In turn, this simple observation is surprisingly powerful, as it allows for the use of recent spectral techniques developed by the first author and collaborators \cite{Abiad2019OnGraphs} to study the higher independence numbers of $G'$ from its eigenvalues. Note that the spectrum of $G$ is not generally related with the spectrum of $G'$~\cite{Abiad2022OptimizationPowers,Das2013LaplacianGraph}, making the approach of considering $G'$ substantially different from (and more feasible than) the one of considering $G$ directly.
In this paper, we focus on two spectral graph theory techniques, which yield
the Inertia-type and 
Ratio-type bounds; see~\cite{Abiad2019OnGraphs} for more details.

The approach we just described has been recently applied to some classes of error-correcting codes, most notably to sum-rank-metric codes~\cite{Abiad2023EigenvalueCodes}, Lee codes \cite{ Abiad2024EigenvalueCodesb} and alternating-rank-metric codes \cite{Abiad2024EigenvalueCodes}, obtaining bounds that often outperform those derived with more traditional arguments, see e.g.~\cite{Huffman2003FundamentalsCodes,byrne2021fundamental}. Eigenvalue bounds like the ones proposed in this paper can also be used to prove that codes meeting a certain bound with equality cannot exist, see e.g.~\cite{Abiad2023EigenvalueCodes,Fiol2020ANumber}. This strongly suggests that spectral graph theory methods can uncover structural properties of ambient spaces that are relevant to coding theory, but that are \textit{not} captured by classical techniques. An example is the sum-rank-metric space~\cite{Abiad2023EigenvalueCodes}, which is a hybrid between rank-metric and Hamming-metric spaces, and for which classical coding theory arguments can lead to quite coarse bounds~\cite{byrne2021fundamental}.

Motivated by these encouraging results, in this paper we investigate the fundamental assumptions underlying the applicability of the Eigenvalue Method in coding theory, and investigate its generality. More precisely, we identify the key compatibility assumptions between the ambient space $A$ and the corresponding graph that allow the application of the spectral graph theory machinery. We then apply these techniques to several ambient spaces and metrics that naturally arise in coding theory, highlighting the cases where the new approach improves on the state of the art. The Eigenvalue Method can be seen as a variation of Delsarte's linear programming (LP) method, but it does not require any regularity on the graph associated to the metric, making it possible to be easily used in cases when Delsarte's method does not apply. While for distance-regular graphs one can use the celebrated linear programming bound by Delsarte on $G^k$, some of the newly proposed eigenvalue bounds are much more general. Indeed, they can also be applied to vertex-transitive graphs which are not distance-regular or, in general, to walk-regular graphs which are not distance-regular. In order to illustrate the applicability and the power of the proposed Eigenvalue Method, we use it to improve on several known results such as \cite[Theorem 13.49]{Berlekamp1968TheCodes} and \cite[Theorem 3.1]{Garcia-Claro2022OnCodes} (city block metric), and \cite{SR2025} (phase-rotation metric), besides the known improvements that the Eigenvalue Method gave on the for the sum-rank metric \cite{Abiad2023EigenvalueCodes} and  the Lee metric \cite{Abiad2024EigenvalueCodesb}. Moreover, by applying the new method we also obtain multiple sharp bounds that give an alternative approach to known results such as \cite{Bossert1996Singleton-TypeCodes} (block metric and cyclic $b$-burst metric), \cite{Borden1983ACodes} and \cite{Varshamov1964EstimateRussian} (Varshamov metric), on top of the known equivalent bounds that the Eigenvalue Method gave for the alternating forms metric \cite{Abiad2024EigenvalueCodes}.

The remainder of this paper is organized as follows. In Section \ref{sec:preliminaries} the necessary preliminaries on coding theory and on graph theory are treated. A description of the Eigenvalue Method is given in Section \ref{sec:evmethod}. This section also contains conditions on the applicability of the method. In Section \ref{sec: spectral bounds} the spectral bounds that are used in the Eigenvalue Method are stated. In order to illustrate the applicability range and power of this newly proposed method, the Eigenvalue Method is applied to several discrete metric spaces. Two of such new applications are discussed in Section \ref{sec:results}. In particular, the method is applied to the city block metric in Section \ref{sec: city block} and to the phase-rotation metric in Section \ref{sec: phase rotation}. A few more applications of the Eigenvalue Method are given in Section \ref{sec: equality metrics}.

\section{Preliminaries}\label{sec:preliminaries}

 In this section, we establish the notation for the rest of the paper and briefly survey the needed background.  
By ``natural numbers'' 
we mean the positive integers, i.e. $\mathbb{N} = \{1,2,3, \ldots \}$. The set of natural numbers with zero is denoted by $\mathbb{N}_0$. For $m \in \mathbb{N}$, let $[m]$ denote the set of integers from $1$ to $m$ and let $[\![m]\!]$ denote the set of integers from $0$ to $m$; $[m] := \{1, \ldots, m\}$ and $[\![m]\!] := \{0,1, \ldots, m\}$. We denote the standard basis vectors of any $n$-dimensional vector space as $\textbf{e}_1, \ldots, \textbf{e}_n$. The all-zeros vector and the all-ones vector in such a vector space are denoted as $\textbf{0}$ and $\textbf{1}$, respectively. 

In this paper we take $m,n \in \mathbb{N}$ and $q$ a prime power, i.e. $q=p^k$ for some prime $p$ and $k \in \mathbb{N}$. The set of integers modulo $m$ is denoted as $\mathbb{Z}/m\mathbb{Z}$. The finite field of $q$ elements is denoted $\Fq$. Moreover, $\Fq^*$ denotes the multiplicative group of nonzero elements of $\Fq$. 

The indicator function of an event $S$ is denoted as $\mathbbm{1}_S$, or as $\mathbbm{1}\{x \in S\} = \mathbbm{1}_S(x)$.

\subsection{Coding theory}
We briefly recall some definitions from coding theory.
A \textit{discrete metric space} is a pair $(\mathcal{X},d)$ where $\mathcal{X}$ is a finite set and $d:\mathcal{X} \times \mathcal{X} \to \R_{\geq 0}$ is a function such that
\begin{itemize}
\item for all $\textbf{x},\textbf{y} \in \mathcal{X}$ we have $d(\textbf{x},\textbf{y})=0 \Leftrightarrow \textbf{x} = \textbf{y}$;
\item for all $\textbf{x},\textbf{y} \in \mathcal{X}$ we have $d(\textbf{x},\textbf{y})=d(\textbf{y},\textbf{x})$;
\item for all $\textbf{x},\textbf{y},\textbf{z} \in \mathcal{X}$ we have $d(\textbf{x},\textbf{z}) \leq d(\textbf{x},\textbf{y}) + d(\textbf{y},\textbf{z})$, which is the triangle inequality.
\end{itemize}
The classic example of a discrete metric space in coding theory is the set $\Fq^n$ with the \textit{Hamming metric} $d_\text{H}$, which is defined as $d_\text{H}(\textbf{x},\textbf{y}) := |\{1 \leq i \leq n: x_i \neq y_i\}|$ for $\textbf{x}=(x_1,\ldots, x_n), \textbf{y} =(y_1, \ldots, y_n) \in \Fq^n$.

A \textit{code} is a subset $\mathcal{C} \subseteq \mathcal{X}$ with $|\mathcal{C}| \geq 2$. The elements of a code $\mathcal{C}$ are called \textit{code words}. The \textit{minimum distance} of a code $\mathcal{C} \subseteq \mathcal{X}$ is defined as
\[
d(\mathcal{C}) := \min \{d(\textbf{x},\textbf{y}) \mid \textbf{x},\textbf{y} \in \mathcal{C}, \, \textbf{x} \neq \textbf{y}\}.
\]
The main problem of classical coding theory is understanding how large a code of certain minimum distance can be. In this regard the largest cardinality of a code $\mathcal{C} \subseteq \Fq^n$ of minimum distance $d$ is denoted as $A_q(n,d)$. For the Hamming metric several upper bounds exist for this quantity $A_q^\text{H}(n,d)$, e.g. the Singleton bound \cite[Theorem 2.4.1]{Huffman2003FundamentalsCodes}, the Hamming bound (or sphere-packing bound) \cite[Theorem 1.12.1]{Huffman2003FundamentalsCodes}, and the Plotkin bound \cite[Theorem 2.2.1]{Huffman2003FundamentalsCodes}. On the other hand, code constructions can give lower bounds for $A_q^\text{H}(n,d)$; see \cite{Huffman2003FundamentalsCodes, MacWilliams1977TheoryCodes} among others.  

While the problem of computing the maximum cardinality of a code with given minimum distance has been extensively studied for the Hamming metric, the question is less understood for other metrics. The Lee metric, the rank metric, and the sum-rank metric, among others, are examples of metrics that are also often used in coding theory. Sum-rank-metric codes, for instance, have been used for multi-shot network coding~\cite{martinez2019reliable} and space-time coding~\cite{shehadeh2021space}. %
More discrete metric spaces follow in the remainder of this paper.

\subsection{Graph theory}
Next we recall some notions of graph theory, with a special focus on the graph properties that are used in the rest of this paper. 

\begin{definition}
A graph is \textit{$\delta$-regular} if every vertex in the graph has degree $\delta$. A graph is said to be \textit{regular} if the graph is $\delta$-regular for some $\delta \in \mathbb{N}_0$.
\end{definition}
A \textit{graph automorphism} of a graph $G=(V,E)$ is a permutation $\sigma$ of the vertex set $V$ such that $(x,y) \in E$ if and only if $(\sigma(x), \sigma(y)) \in E$. A graph is \textit{vertex-transitive} if for any two vertices $x,y$ there exists a graph automorphism $\sigma$ such that $\sigma(x)=y$. Note that vertex-transitive graphs are regular.

A graph is a \textit{Cayley graph} over a group $G$ with connecting set $S$ if the vertices of the graph are the elements of $G$, and two vertices $x,y$ are adjacent if and only if there is an element $s \in S$ such that $x+s=y$. In this work we assume that the connecting set $S$ does not contain the identity element of $G$ and that $S$ is closed under inverses. This assumption implies that the corresponding Cayley graph is undirected and has no self-loops. Note that Cayley graphs are vertex-transitive.

\begin{definition}
A graph is \textit{$k$-partially walk-regular} if for any vertex $x$ and any positive integer $i \leq k$ the number of closed walks of length $i$ that start and end in $x$ does not depend on the choice of $x$. A graph is \textit{walk-regular} if it is $k$-partially walk-regular for any positive integer~$k$.
\end{definition}
Note that vertex-transitive graphs are necessarily walk-regular.

For any two vertices $x$ and $y$ at distance $i$ from each other, let $p_{j,h}^i(x,y)$ denote the number of vertices at distance $j$ from $x$ and at distance $h$ from $y$. 
\begin{definition} 
A graph is \textit{$k$-partially distance-regular} if for any integers $i,j,h$ such that $j,h \leq k$ and $i \leq j+h \leq k$ the values $p_{j,h}^i(x,y)$ do not depend on the choice of $x$ and $y$. A graph is \textit{distance-regular} if it is $k$-partially distance-regular for any integer $k$.
\end{definition}
In particular, a graph is $k$-partially distance-regular if for any integer $i \leq k$ the values $c_i(x,y):=p_{1,i-1}^i(x,y)$, $a_{i-1}(x,y):=p_{1,i-1}^{i-1}(x,y)$, and $b_{i-2} := p_{1,i-1}^{i-2}(x,y)$ do not depend on the choice of $x$ and $y$. For distance-regular graphs these values are captured in the \textit{intersection array} $(b_0, b_1, \ldots, b_{D-1}; c_1, \ldots, c_D)$, where $D$ is the diameter of the graph. Since distance-regular graphs are $\delta$-regular for some $\delta \in \mathbb{N}_0$, the following relations hold: $a_i+b_i+c_i=\delta$ for $0 \leq i \leq D$, $b_0=\delta$, $a_0=c_0=0$. Note that $k$-partially distance-regular graphs are also $k$-partially walk-regular. 


Recall the \textit{Cartesian product} of two graphs $G$ and $H$, denoted as $G \Box H$, which is the graph with vertex set equal to the Cartesian product of the vertex sets of $G$ and $H$, where two vertices $(g_1,h_1)$ and $(g_2,h_2)$ are adjacent if $g_1 \sim g_2$ and $h_1=h_2$, or $g_1=g_2$ and $h_1 \sim h_2$. Here $\sim$ denotes adjacency of the vertices in the graph. The Cartesian product of two graphs can be inductively extended to a Cartesian product of finitely many graphs. It is well-known that if $G$ and $H$ are graphs with respective eigenvalues $\lambda_i, i \in I$ and $\mu_j, j \in J$, then the eigenvalues of $G \Box H$ are $\lambda_i + \mu_j$ for $i \in I, j \in J$ (see for instance \cite{Cvetkovic1980SpectraApplications}). This can be inductively extended to the Cartesian product of finitely many graphs.

\begin{definition}
The \textit{$k$-independence number} of a graph $G$, denoted as $\alpha_k(G)$, is the size of the largest set of vertices in $G$ such that any two vertices in the set are at geodesic distance greater than $k$ from each other. 
\end{definition}
Alternatively we can consider the \textit{$k$-th power graph} $G^k$ of a graph $G=(V,E)$, which is the graph with vertex set $V$ where two vertices $x,y \in V$ are adjacent if $d_G(x,y) \leq k$. Here $d_G(x,y)$ denotes the geodesic distance between vertices $x$ and $y$ in the graph~$G$. The $k$-independence number of $G$ equals the ($1$-)independence number of $G^k$, which is the size of the largest independent set in $G^k$. Despite this, even the simplest algebraic or combinatorial parameters of the power graph $G^k$ cannot be easily deduced from the corresponding parameters of the graph $G$, e.g. neither the spectrum~\cite{Das2013LaplacianGraph},~\cite{Abiad2022OptimizationPowers}, nor the average degree~\cite{Devos2013AveragePowers}, nor the rainbow connection number~\cite{Basavaraju2014RainbowProducts} of $G^k$ can be derived in general directly from those of the original graph $G$. 
In this regard, several eigenvalue bounds on $\alpha_k(G)$ that only depend on the spectrum of $G$ have been proposed in the literature. Another upper bound on the independence number, and after extension the $k$-independence number, of a graph is the \textit{Lov\'asz theta number} \cite{Lovasz1979OnGraph}, although this bound requires the graph adjacency matrix as input. The Lov\'asz theta number can be used as an upper bound on the $k$-independence number of a graph $G$ by computing it for the $k$-th power graph $G^k$.

The $k$-independence number of a graph and these eigenvalue bounds are the bases on which the Eigenvalue Method is built, as we see in the next section.

\section{The Eigenvalue Method}\label{sec:evmethod}

In this section we give a description of the Eigenvalue Method and we give conditions on the applicability of the method. In later sections applications of the Eigenvalue Method are discussed.

As introduced earlier, there is a natural connection between coding theory and graph theory, which enables the use of bounds on the $k$-independence number for the construction of bounds on the cardinality of codes with given correction capability. The method can be formalized as follows. Let $(\mathcal{X},d)$ be a discrete metric space. Define the distance graph $G_d(\mathcal{X})$ for $(\mathcal{X}, d)$ as the graph with vertex set $\mathcal{X}$ where vertices $x,y \in \mathcal{X}$ are adjacent if $d(x,y)=1$. If the geodesic distance between vertices in $G_d(\mathcal{X})$ equals the distance between corresponding elements in the discrete metric space, then there is an equivalence between the maximum cardinality of codes in $(\mathcal{X}, d)$ and the $k$-independence number of $G_d(\mathcal{X})$. The next result formalizes this equivalence.

\begin{lemma}[\text{\cite[Corollary 16]{Abiad2023EigenvalueCodes}}] \label{lemma: geo dist then max card alphak}
Let $(\mathcal{X},d)$ be a discrete metric space. Suppose the geodesic distance in $G_d(\mathcal{X})$ equals the distance in the discrete metric space $(\mathcal{X},d)$, i.e. $d_{G_d(\mathcal{X})}(x,y) = d(x,y)$ for all $x,y \in \mathcal{X}$. Then the maximum cardinality of a code $\mathcal{C} \subseteq \mathcal{X}$ of minimum distance~$d'$ equals the $k$-independence number of $G_d(\mathcal{X})$ for $k=d'-1$, namely $\alpha_{d'-1}(G_d(\mathcal{X}))$. 
\end{lemma}

Bounds on the $k$-independence number can now be used to obtain bounds on the cardinality of codes. Specifically, we consider two spectral bounds for the $k$-independence number, namely the Inertia-type bound and the Ratio-type bound, which are the main tools of the Eigenvalue Method. These spectral bounds can be found in Section \ref{sec: spectral bounds}, together with their respective linear programming implementations. The graph $G_d(\mathcal{X})$ should have certain graph properties for these spectral bounds to be applicable to the graph.

For the Inertia-type bound from Theorem \ref{th:inertia type bound} and corresponding mixed-integer linear program (MILP) \eqref{eq: milp inertia} there are no extra graph properties that $G_d(\mathcal{X})$ needs to have. This MILP requires as input the adjacency matrix of the graph besides the graph adjacency spectrum. A faster MILP for the Inertia-type bound which only requires the graph adjacency spectrum as input is MILP \eqref{eq: milp inertia walk reg}; this MILP only works for $k$-partially walk-regular graphs, so it is desirable for $G_d(\mathcal{X})$ to have this property. The Ratio-type bound from Theorem \ref{th:ratio type bound} only applies to regular graphs, so regularity of $G_d(\mathcal{X})$ is preferred. For the linear program (LP) \eqref{eq: lp ratio}, which corresponds to the Ratio-type bound, the input graph is required to be $k$-partially walk-regular. So $k$-partial walk-regularity of $G_d(\mathcal{X})$ is preferred here as well. Note that this LP only needs the graph adjacency spectrum as input.

We compare the Eigenvalue Method to Delsarte's linear programming method. In general it is not known if bounds obtained via Delsarte's method are stronger than bounds obtained using the Inertia-type bound or the Ratio-type bound. However, since Delsarte's LP method directly applies when $G_d(\mathcal{X})$ is distance-regular, we prefer to restrict to discrete metric spaces where the corresponding graph $G_d(\mathcal{X})$ is not distance-regular. 

The only necessary condition for the applicability of the Eigenvalue Method is the following:
\begin{enumerate}[label=(C\arabic*)]
\item The geodesic distance in $G_d(\mathcal{X})$ equals the distance in the discrete metric space $(\mathcal{X},d)$, i.e., $d_{G_d(\mathcal{X})}(x,y)=d(x,y)$ for all $x,y \in \mathcal{X}$. \label{item: C1}
\end{enumerate}


Moreover, some graph properties of $G_d(\mathcal{X})$ are highly desired. These can be summarized as follows.
\begin{enumerate}[label=(P\arabic*)]
\item The graph $G_d(\mathcal{X})$ is regular. This property is desirable as the Ratio-type bound applies if this is the case. \label{item: C2}
\item The graph $G_d(\mathcal{X})$ is $k$-partially walk-regular. This property is desirable as the faster MILP implementation of the Inertia-type bound and the LP implementation of the Ratio-type bound apply if this is the case. \label{item: C3}
\item The graph $G_d(\mathcal{X})$ is not distance-regular. This property is desirable as Delsarte's LP method is not directly applicable if this is the case. \label{item: C4}
\end{enumerate}
\subsection{Eigenvalue bounds} \label{sec: spectral bounds}
In this section we give the eigenvalue bounds that are used in the Eigenvalue Method. First the Inertia-type bound and its MILP implementation are given. Then the Ratio-type bound and its LP implementation are stated.

Define $\mathbb{R}_k[x]$ as the set of all polynomials in the variable $x$ with real coefficients and degree at most $k$. The Inertia-type bound is an upper bound on the $k$-independence number of a graph. 

\begin{theorem}[\text{Inertia-type bound, \cite[Theorem 3.1]{Abiad2019OnGraphs}}] \label{th:inertia type bound}
Let $G$ be a graph with $n$ vertices, adjacency eigenvalues $\lambda_1 \geq \cdots \geq \lambda_n$, and adjacency matrix $A$. Let $p \in \mathbb{R}_k[x]$ with corresponding parameters $W(p):= \max_{u \in V(G)} \{(p(A))_{uu}\}$, $w(p):= \min_{u \in V(G)} \{(p(A))_{uu}\}$. Then the $k$-independence number $\alpha_k$ of $G$ satisfies
\[
\alpha_k \leq \min \left\{|\{i: p(\lambda_i) \geq w(p)\}|, \, |\{i: p(\lambda_i) \leq W(p)\}| \right\}.
\]
\end{theorem}

Note that for $k=1$ the Inertia-type bound reduces to the well-known inertia bound by Cvetkovi\'c \cite{Cvetkovic1971GraphsSpectra}.


In \cite{Abiad2022OptimizationPowers} an MILP has been proposed that finds the optimal polynomial for the Inertia-type bound, which is the polynomial that minimizes the upper bound on the $k$-independence number. This MILP subsequently finds this minimized upper bound on the $k$-independence number. Let $G$ be a graph with $n$ vertices, distinct adjacency eigenvalues $\theta_0 > \cdots > \theta_r$ with respective multiplicities $m_0, \ldots, m_r$, and adjacency matrix $A$. Let $p(x) := a_kx^k + \cdots + a_0$, $\textbf{b}=(b_0, \ldots, b_r) \in \{0,1\}^{r+1}$, and $\textbf{m} = (m_0, \ldots, m_r)$. Then the following MILP with variables $a_0, \ldots, a_k$ and $b_0, \ldots, b_r$ finds the optimal polynomial for Theorem \ref{th:inertia type bound}:

\begin{center}
\fbox{\parbox{\linewidth}{
\begin{equation} \label{eq: milp inertia}
\begin{aligned}
\text{minimize } &\textbf{m}^\top \textbf{b} \\
\text{subject to } &\sum_{i=0}^k a_i (A^i)_{vv} \geq 0, &v&\in V(G) \backslash \{u\} \\
&\sum_{i=0}^k a_i (A^i)_{uu} = 0 & \\
&\sum_{i=0}^k a_i \theta_j^i - M b_j + \epsilon \leq 0, &j &= 0, \ldots, d \\
&\textbf{b} \in \{0,1\}^{r+1}
\end{aligned}
\end{equation}}}
\end{center}

Here $M$ is some fixed large number and $\epsilon>0$ is small. This MILP has to run for every $u \in V(G)$ and the best objective value is then the minimum of all separate objective values. This lowest objective value is exactly the best upper bound for the $k$-independence number that can be obtained from the Inertia-type bound.

In \cite{Abiad2022OptimizationPowers} it is discussed that if the graph $G$ is $k$-partially walk-regular, then MILP \eqref{eq: milp inertia} only has to run for one $u \in V(G)$. In this case, using the same notation as above, MILP \eqref{eq: milp inertia} simplifies to the following MILP:

\begin{center}
\fbox{\parbox{\linewidth}{
\begin{equation} \label{eq: milp inertia walk reg}
\begin{aligned}
\text{minimize } & \textbf{m}^\top \textbf{b} & \\
\text{subject to } & \sum_{i=0}^r m_i p(\theta_i) = 0 & \\
& \sum_{i=0}^k a_i \theta_j^i - M b_j + \epsilon \leq 0, &j &= 0, \ldots, d \\
& \textbf{b} \in \{0,1\}^{r+1}
\end{aligned}
\end{equation}}}
\end{center}

For $k$-partially walk-regular graphs $G$, the objective value of MILP \eqref{eq: milp inertia walk reg} is exactly the best upper bound for the $k$-independence number that can be obtained from the Inertia-type bound.

The Ratio-type bound is another upper bound on the $k$-independence number, but specifically for regular graphs.

\begin{theorem}[\text{Ratio-type bound, \cite[Theorem 3.2]{Abiad2019OnGraphs}}] \label{th:ratio type bound}
Let $G$ be a regular graph with $n$ vertices, adjacency eigenvalues $\lambda_1 \geq \cdots \geq \lambda_n$, and adjacency matrix $A$. Let $p \in \mathbb{R}_k[x]$ with corresponding parameters $W(p):= \max_{u \in V(G)} \{(p(A))_{uu}\}$, $\lambda(p):= \min_{i \in [2,n]} \{p(\lambda_i)\}$. Assume that $p(\lambda_1) > \lambda(p)$. Then the $k$-independence number $\alpha_k$ of $G$ satisfies
\[
\alpha_k \leq n \frac{W(p)-\lambda(p)}{p(\lambda_1) - \lambda(p)}.
\]
\end{theorem}

For $k=1$ the Ratio-type bound can be reduced to the well-known ratio bound by Hoffman (unpublished, see e.g. \cite[Theorem 3.2]{Haemers1995InterlacingGraphs}).


For $k=2,3$ there are closed-form expressions for the Ratio-type bound that no longer depend on the choice of $p \in \mathbb{R}_k[x]$ and that are optimal in the sense that no better bound can be obtained via Theorem \ref{th:ratio type bound}. 

\begin{theorem}[\text{\cite[Corollary 3.3]{Abiad2019OnGraphs}}] \label{th: alpha_2}
Let $G$ be a regular graph with $n$ vertices and distinct adjacency eigenvalues $\theta_0 > \theta_1 > \cdots > \theta_r$ with $r \geq 2$. Let $\theta_i$ be the largest eigenvalue such that $\theta_i \leq -1$. Then the 2-independence number $\alpha_2$ of $G$ satisfies
\[
\alpha_2 \leq n \frac{\theta_0 + \theta_i \theta_{i-1}}{(\theta_0 - \theta_i)(\theta_0 - \theta_{i-1})}.
\]
Moreover, this is the best possible bound that can be obtained by choosing a polynomial via Theorem \ref{th:ratio type bound}.
\end{theorem}

\begin{theorem}[\text{\cite[Theorem 11]{Kavi2023TheMethod}}] \label{th: alpha3}
Let $G$ be a regular graph with $n$ vertices, distinct adjacency eigenvalues $\theta_0> \theta_1 > \cdots > \theta_r$ with $r \geq 3$, and adjacency matrix $A$. Let $\theta_s$ be the smallest eigenvalue such that $\theta_s \geq - \frac{\theta_0^2 + \theta_0 \theta_r - \Delta}{\theta_0 (\theta_r+1)}$, where $\Delta = \max_{u \in V(G)} \{(A^3)_{uu}\}$. Then the 3-independence number $\alpha_3$ of $G$ satisfies
\[
\alpha_3 \leq n \frac{\Delta - \theta_0 (\theta_s + \theta_{s+1} + \theta_r) - \theta_s \theta_{s+1} \theta_r}{(\theta_0 - \theta_s)(\theta_0 - \theta_{s+1})(\theta_0 - \theta_r)}.
\]
Moreover, this is the best possible bound that can be obtained by choosing a polynomial via Theorem \ref{th:ratio type bound}.
\end{theorem}

In \cite{Fiol2020ANumber} an LP has been proposed that finds the optimal polynomial for the Ratio-type bound, which is the polynomial that minimizes the upper bound on the $k$-independence number, in case the graph $G$ is $k$-partially walk-regular. The objective value of this LP, which uses the socalled minor polynomials, subsequently equals this minimized upper bound. Let $G$ be a $k$-partially walk-regular graph with distinct adjacency eigenvalues $\theta_0 > \theta_1 > \cdots > \theta_r$ with respective multiplicities $m_0, m_1, \ldots, m_r$. The \textit{$k$-minor polynomial} is the polynomial $f_k \in \mathbb{R}_k[x]$ that minimizes $\sum_{i=0}^r m_i f(\theta_i)$. Define the polynomial $f_k$ as $f_k(\theta_0) := x_0 = 1$ and $f_k(\theta_i) := x_i$ for $i=1, \ldots, r$, where $(x_1, \ldots, x_r)$ is a solution of the following LP:

\begin{center}
\fbox{\parbox{\linewidth}{
\begin{equation} \label{eq: lp ratio}
\begin{aligned}
\text{minimize } & \sum_{i=0}^r m_ix_i & \\
\text{subject to } & f[\theta_0, \ldots. \theta_s] = 0, &s&=k+1, \ldots, r \\
& x_i \geq 0, &i&=1, \ldots, r  
\end{aligned}
\end{equation}}}
\end{center}

Here $f[\theta_0, \ldots, \theta_s]$ denotes the \textit{$s$-th divided difference of Newton interpolation}, recursively defined by
\[
f[\theta_i, \ldots, \theta_j] := \frac{f[\theta_{i+1}, \ldots, \theta_j] - f[\theta_i, \ldots, \theta_{j-1}]}{\theta_j - \theta_i},
\]
where $j>i$, starting with $f[\theta_i] = f(\theta_i) = x_i$ for $i=0,1,\ldots, r$. The best upper bound for the $k$-independence number of $k$-partially walk-regular graphs that can be obtained via the Ratio-type bound is exactly the objective value of LP \eqref{eq: lp ratio}, which follows from \cite[Theorem 4.1]{Fiol2020ANumber}.

Note that the Lovász theta number is at least as good as the Ratio-type bound, but the Ratio-type bound can be computed exactly and more efficiently since it only requires the graph spectrum, while the Lovász theta number is an approximation and requires the graph adjacency matrix. The relation between the performance of the Lovász theta number and the Inertia-type bound is not known.

\section{Improved bounds in several metrics}\label{sec:results}

In this section we apply the Eigenvalue Method to some discrete metric spaces to estimate the size of codes. The results show that it is a new powerful tool for coding theory. Indeed, the bounds obtained using the Eigenvalue Method turn out to improve on state of the art upper bounds on the cardinality of codes in those discrete metric spaces. See Table \ref{tab:metrics overview} for an overview of the metrics already considered in literature and considered in the remainder of this paper. In this section we consider discrete metric spaces with the following distance functions: the city block metric, the projective metric (which is a class of distance functions that includes well-known metrics such as the Hamming metric, the rank metric and the sum-rank metric), and the phase-rotation metric.

\begin{table}[ht] 
\centering
\begin{tabular}{|l l|c|c|}
\hline
\multicolumn{2}{|c|}{Metric} & Sharp & Improvement \\
\hline 
\hline
Alternating rank & \cite{Abiad2024EigenvalueCodes} & Ratio-type & - \\
\hline
Block & Section~\ref{sec: block} & Ratio-type & - \\
\hline
City block & Section~\ref{sec: city block} & Inertia-type & Inertia-type \\ 
\hline
Cyclic $b$-burst & Section~\ref{sec: cyclic b burst} & Ratio-type & - \\ 
\hline
Lee & \cite{Abiad2024EigenvalueCodesb} & Ratio-type & Ratio-type \\
\hline
Phase-rotation & Section~\ref{sec: phase rotation} & Inertia-type, Ratio-type & Inertia-type, Ratio-type \\ 
\hline
Sum-rank & \cite{Abiad2023EigenvalueCodes} & Ratio-type & Ratio-type \\
\hline
Varshamov & Section~\ref{sec: varshamov} & Inertia-type & - \\ 
\hline
\end{tabular}
\caption{Overview of the metrics studied in literature and in this paper in the context of the Eigenvalue Method. If one of the proposed spectral bounds is sharp in some instances, this is indicated in the column ``Sharp''. If a spectral bound gives an improvement compared to the state of the art bounds in some instances, this is indicated in the column ``Improvement''.}
\label{tab:metrics overview}
\end{table}

\subsection{City block metric} \label{sec: city block}
The city block metric, also called the $L^1$-metric or the Manhattan metric, was used already in the 18th century. Its first appearance in a coding-theoretical context is in \cite{Ulrich1957NonBinaryCodes}, although it is not properly defined as a metric yet. The city block metric can be viewed as an extension of the Lee metric, which is widely used in coding theory, to $\mathbb{Z}^n$ instead of $\Fq^n$.

\begin{definition}
The \textit{city block distance} between $\textbf{x}=(x_1,\ldots, x_n),\textbf{y}=(y_1,\ldots,y_n) \in [\![m-1]\!]^n$ is defined as $d_\text{cb}(\textbf{x},\textbf{y}) := \sum_{i=1}^n |x_i-y_i|$.
\end{definition} 


Note that for $m=2$ the city block metric coincides with the Hamming metric. Therefore we assume $m \geq 3$. 

Now we consider the discrete metric space $([\![m-1]\!]^n, d_\text{cb})$ and apply the Eigenvalue Method. Define the \textit{city block distance graph} $G_\text{cb}([\![m-1]\!]^n)$ as the graph with vertex set $[\![m-1]\!]^n$ where vertices $\textbf{x},\textbf{y} \in [\![m-1]\!]^n$ are adjacent if $d_\text{cb}(\textbf{x},\textbf{y})=1$. First we verify that condition \ref{item: C1} holds for this graph.

\begin{lemma}
The geodesic distance in $G_\text{cb}(\mathbb{Z}_m^n)$ equals the city block distance.
\end{lemma}

\begin{proof}
Let $\textbf{x}=(x_1,\ldots,x_n),\textbf{y}=(y_1,\ldots,y_n) \in [\![m-1]\!]^n$. Define $r_i := |x_i-y_i|$ for $i=1, \ldots, n$. Then $d_\text{cb}(\textbf{x},\textbf{y}) = \sum_{i=1}^n r_i$. We can make a path in $G_\text{cb}(\mathbb{Z}_m^n)$ from $\textbf{x}$ to the vertex where the first coordinate is replaced by $y_1$ of length $r_1$ by going to a neighboring vertex which has 1 added to (or subtracted from) the first coordinate until we reach the desired vertex $(y_1, x_2, \dots, x_n)$.
Similarly, we can make a path in $G_\text{cb}([\![m-1]\!]^n)$ from $(y_1, x_2, \ldots, x_n)$ to the vertex where the second coordinate is replaced by $y_2$ of length $r_2$, and so on for all other coordinates. Traversing these paths one after another gives a path of length $\sum_{i=1}^n r_i$ from vertex $\textbf{x}$ to vertex $\textbf{y}$ in $G_\text{cb}([\![m-1]\!]^n)$. So the geodesic distance from $\textbf{x}$ to $\textbf{y}$ is at most $\sum_{i=1}^n r_i$. 

If the geodesic distance is less than $\sum_{i=1}^n r_i$, then, using the same path construction as before, it follows from the triangle inequality that $d_\text{cb}(\textbf{x},\textbf{y}) < \sum_{i=1}^n r_i$. This is a contradiction, so the geodesic distance in $G_\text{cb}([\![m-1]\!]^n)$ equals the city block distance.
\end{proof}

Since condition \ref{item: C1} holds, the Eigenvalue Method is applicable to the discrete metric space $([\![m-1]\!]^n, d_\text{cb})$. Next we check the desired properties \ref{item: C2}, \ref{item: C3}, and \ref{item: C4}.

\begin{remark} \label{remark: Gcb not reg}
The graph $G_\text{cb}([\![m-1]\!]^n)$ is not regular. The neighbors of $\textbf{0}$ are the vectors $\textbf{x}=(x_1,\ldots,x_n) \in [\![m-1]\!]^n$ such that $x_i = 1$ for exactly one $i \in \{1, \ldots, n\}$ and $x_j=0$ for all $j \neq i$. So $\textbf{0}$ has $n$ neighbors. The neighbors of $\textbf{1}$ are the vectors $\textbf{x}=(x_1,\ldots,x_n) \in [\![m-1]\!]^n$ such that $x_i = 0$ or $x_i = 2$ for exactly one $i \in \{1, \ldots, n\}$ and $x_j = 1$ for all $j \neq i$. So $\textbf{1}$ has $2n$ neighbors, implying that $G_\text{cb}([\![m-1]\!]^n)$ is not regular. Hence $G_\text{cb}([\![m-1]\!]^n)$ is also not walk-regular nor distance-regular. 
\end{remark}

Remark \ref{remark: Gcb not reg} shows that properties \ref{item: C2} and \ref{item: C3} are not satisfied, while \ref{item: C4} is. This implies that only the Inertia-type bound is applicable to the graph $G_\text{cb}([\![m-1]\!]^n)$ and not the Ratio-type bound as it requires regularity of the graph. In order to apply the former bound, we first determine the adjacency eigenvalues of the graph. These eigenvalues follow directly from the next result.

\begin{lemma}\label{lemma: cart prod path graphs cb}
The graph $G_\text{cb}([\![m-1]\!]^n)$ equals the Cartesian product of $n$ path graphs on $m$ vertices.
\end{lemma}


\begin{proof}[Proof of Lemma \ref{lemma: cart prod path graphs cb}]
Fix $m$. We prove the result by induction on $n$. 

For $n=1$ the graph $G_\text{cb}([\![m-1]\!]^n)$ has $m$ vertices indexed as $0,1, \ldots, m-1$ and edge set $\{(i,i+1): i=0, \ldots, m-2\}$. So $G_\text{cb}([\![m-1]\!]^n)$ indeed equals the path graph on $m$ vertices (after renumbering of the vertices).

Now suppose that the city block distance graph for the discrete metric space $([\![m-1]\!]^n, d_\text{cb})$ equals the Cartesian product of $n$ path graphs on $m$ vertices: 
\[
\underbrace{P_m \Box \cdots \Box P_m}_{n \text{ times}} =: G,
\]
where $P_m$ denotes the path graph on $m$ vertices. Consider the city block distance graph $G_\text{cb}([\![m-1]\!]^{n+1})$, so for $n+1$. The vertex set of $G \Box P_m$ equals the vertex set of $G_\text{cb}([\![m-1]\!]^{n+1})$ (given the right naming of the vertices of $P_m$, namely 0 through $m-1$). Let $\textbf{x}=(x_1,\ldots, x_{n+1}),\textbf{y}=(y_1,\ldots, y_{n+1}) \in [\![m-1]\!]^{n+1}$ be two adjacent vertices in $G \Box P_m$. Then 
\begin{itemize}
\item $(x_1, \ldots, x_n) = (y_1, \ldots, y_n)$ and $x_{n+1} \sim y_{n+1}$ so $d_\text{cb}((x_1, \ldots, x_n), (y_1, \ldots, x_n) = 0$ and $d_\text{cb}(x_{n+1}, y_{n+1}) = 1$, or
\item $(x_1, \ldots, x_n) \sim (y_1, \ldots, y_n)$ and $x_{n+1} = y_{n+1}$ so $d_\text{cb}((x_1, \ldots, x_n), (y_1, \ldots, x_n) = 1$ and $d_\text{cb}(x_{n+1}, y_{n+1}) = 0$. 
\end{itemize}
In both cases $d_\text{cb}(\textbf{x},\textbf{y}) = 1$, so $\textbf{x}$ and $\textbf{y}$ are adjacent in $G_\text{cb}([\![m-1]\!]^{n+1})$. Moreover, these are the only two options when $\textbf{x}, \textbf{y} \in [\![m-1]\!]^{n+1}$ are adjacent in $G_\text{cb}([\![m-1]\!]^{n+1})$. Hence the graph $G_\text{cb}([\![m-1]\!]^{n+1})$ is equal to the Cartesian product of graphs $G$ and $P_m$. By induction, $G_\text{cb}([\![m-1]\!]^{n})$ equals the Cartesian product of~$n$ path graphs on $m$ vertices.
\end{proof}

Now we are ready to derive the eigenvalue of our graph of interest.

\begin{lemma}
The adjacency eigenvalues of $G_\text{cb}([\![m-1]\!]^n)$ are
\[
\lambda_\textbf{k} = \sum_{j=1}^n 2\cos \left(\frac{k_j \pi}{m+1} \right)
\]
for every tuple $\textbf{k} = (k_1, \ldots, k_n) \in [m]^n$.
\end{lemma}

\begin{proof}
The adjacency eigenvalues of the path graph $P_m$ are given by $\lambda_k = 2 \cos \left( \frac{k \pi}{m+1} \right)$ for $k \in [m]$ (see for instance \cite{Cvetkovic2009ApplicationsLiterature}). Since $G_\text{cb}([\![m-1]\!]^n)$ equals the Cartesian product of $n$ path graphs $P_m$, the result follows.
\end{proof}

This expression for the eigenvalues of the city block distance graph can now be used in the Inertia-type bound to obtain new bounds on the maximum cardinality of codes in the city block metric. We then compare the obtained bounds to state of the art bounds: the Plotkin-type bound and the Hamming-type bound. 

\begin{theorem}[\text{Plotkin-type bound, \cite[Theorem 13.49]{Berlekamp1968TheCodes}}] \label{th: plotkin city block}
Let $\mathcal{C} \subseteq [\![m-1]\!]^n$ be a code of minimum city block distance $d$. If $d> \frac{n(m-1)}{2}$, then 
\[
|\mathcal{C}| \leq \frac{2d}{2d-n(m-1)}.
\]
\end{theorem}

\begin{remark}
The Plotkin-type bound does not follow immediately from \cite[Theorem 13.49]{Berlekamp1968TheCodes}. After the substitution $\Bar{D} := \frac{1}{m} \sum_{i=1}^{m-1} d_\text{cb}(0,i) = \frac{m-1}{2}$, the bound follows from rewriting the inequality in \cite[Theorem 13.49]{Berlekamp1968TheCodes}.
\end{remark}

\begin{theorem}[\text{Hamming-type bound, \cite[Theorem 3.1]{Garcia-Claro2022OnCodes}}] \label{th: Hamming city block}
Let $\mathcal{C} \subseteq [\![m-1]\!]^n$ be a code of minimum city block distance $d$. Define $t:= \lfloor \tfrac{d-1}{2} \rfloor$. Then
\[
|\mathcal{C} | \leq \frac{m^n}{\eta_t \left( [\![m-1]\!]^n \right) },
\]
where $\eta_t([\![m-1]\!]^n) := \min \left\{ |B_t(\textbf{x})|: \textbf{x} \in [\![m-1]\!]^n \right\}$ and $B_t(\textbf{x}) := \left\{\textbf{y} \in [\![m-1]\!]^n: d_\text{cb}(\textbf{x},\textbf{y}) \leq t\right\}$.
\end{theorem}

In \cite[Algorithm 1]{Garcia-Claro2022OnCodes} an algorithm is described that computes the value of $\eta_t([\![m-1]\!]^n)$ given specific values of $t$ and $\textbf{x}$. We use this algorithm to compute the Hamming-type upper bound in specific instances.

Next, we compare the bound obtained using the Inertia-type bound with the Plotkin-type bound from Theorem \ref{th: plotkin city block} and the Hamming-type bound from Theorem \ref{th: Hamming city block}. We split the discussion in two cases. 

\textbf{Case $k=1$:} 
For $k=1$ the Inertia-type bound reduces to the well-known inertia bound, which is independent of the choice of the polynomial $p \in \R_k[x]$. Therefore, the case $k=1$ is considered first. In this case $d=k+1=2$, so for the Hamming-type bound we get $t:=\lfloor \tfrac{d-1}{2} \rfloor =0$ and $\eta_t([\![m-1]\!]^n) = 1$ since any ball around a code word of radius 0 contains only the code word itself. The Hamming-type bound then becomes $|\mathcal{C}| \leq m^n$, which is a trivial upper bound, so the inertia bound certainly performs better than the Hamming-type bound. Since $d=2$, the condition of the Plotkin-type bound, $d>\tfrac{n(m-1)}{2}$, together with the constraint $m \geq 3$, gives only two instances where the Plotkin-type bound applies: $n=1, m=3$ and $n=1,m=4$. For $n=1, m=3$ both the inertia bound and the Plotkin-type bound give an upper bound of 2. For $n=1,m=4$ the inertia bound gives 2, while the Plotkin-type bound gives 4. So in this instance the inertia bound gives an improved upper bound. Note that in both instances the inertia bound turns out to be sharp. All in all, the inertia bound, which is the Inertia-type bound in the case $k=1$, performs no worse than the Plotkin-type bound and the Hamming-type bound.

\textbf{Case $k \geq 2$:}
Next we consider the case $k \geq 2$. Here we resort to the proposed MILP for computing the value of the Inertia-type bound in specific instances. Since the city block distance graph is not walk-regular, only the slower MILP \eqref{eq: milp inertia} is applicable. This MILP requires the construction of the graph $G_\text{cb}([\![m-1]\!]^n)$ to use the adjacency matrix of this graph. In this case we also compute the Lovász theta number of the graph. 
 
Now we compare the upper bounds from the Inertia-type bound with the Plotkin-type bound, the Hamming-type bound and the Lovász theta number for the following instances:
\[
n =1,2,3, \quad m =3,4, \quad k=1,\ldots, n(m-1)-1.
\]
\[
\text{and } n=1,2, \quad m=5,6, \quad k=1, \ldots n(m-1)-1.
\]
\[
\text{and } n=3, \quad m=5, \quad k=1,\ldots,7.
\]
Note that we are also considering some instances where $k=1$ since we have only seen two of those thus far. The results can be seen in Table \ref{tab: city block results}. The column ``Inertia-type'' gives the output of the Inertia-type bound for the given graph instance. Similarly the column ``$\vartheta(G^k)$'' contains the value of the Lovász theta number, the column ``Plotkin-type'' contains the value of the Plotkin-type upper bound, and the column ``Hamming-type'' contains the value of the Hamming-type bound. The column ``$\alpha_k$'' contains the value of the $k$-independence number of the graph for that instance. Only the instances where the Inertia-type bound performed no worse than the Plotkin-type bound and the Hamming-type bound are present in the table. An upper bound in the column ``Inertia-type'' is marked in bold when it is less than both the Plotkin-type upper bound (if applicable) and the Hamming-type upper bound.

\begin{table}[ht]
\centering
\begin{tabular}{|ccc|c|c|c|c|c|}
\hline
$m$ & $n$ & $k$ & Inertia-type & $\alpha_k$ & $\vartheta(G^k)$ & \makecell{Plotkin-type \\ (Theorem \ref{th: plotkin city block})} & \makecell{Hamming-type \\ (Theorem \ref{th: Hamming city block})} \\
\hline
3 & 1 & 1 &2 & 2 & 2.0 & 2 & 3 \\
4 & 1 & 1 & \textbf{2} & 2 & 2.0 & 4 & 4 \\
4 & 1 & 2 & 2 & 2 & 2.0 & 2 & 2 \\
5 & 1 & 1 & \textbf{3} & 3 & 3.0 & - & 5 \\
5 & 1 & 2 & 2 & 2 & 2.0 & 3 & $\tfrac{5}{2}$ \\
5 & 1 & 3 & 2 & 2 & 2.0 & 2 & $\tfrac{5}{2}$ \\
6 & 1 & 1 & \textbf{3} & 3 & 3.0 & - & 6 \\
6 & 1 & 2 & \textbf{2} & 2 & 2.0 & 6 & 3 \\
6 & 1 & 3 & 2 & 2 & 2.0 & 2 & 3 \\
6 & 1 & 4 & 2 & 2 & 2.0 & 2 & 2 \\
3&2&2&3&2&2.33&3&3\\
4&2&3&\textbf{3}&3&3.0&4&$\tfrac{16}{3}$\\
5&2&4&4&3&3.06&5&$\tfrac{25}{6}$\\
5&2&5&3&2&2.33&3&$\tfrac{25}{6}$\\
6&2&5&6&3&3.17&6&6\\
3&3&3&4&4&4.0&4&$\tfrac{27}{4}$\\
4&3&1&\textbf{32}&32&32.0&-&64\\
4&3&5&4&4&4.0&4&$\tfrac{32}{5}$\\
4&3&8&2&2&2.0&2&2\\

\hline
\end{tabular}
\caption{Results of the Inertia-type bound for the city block metric, compared to the Plotkin-type bound, the Hamming-type bound, the Lovász theta number $\vartheta(G^k)$, and the actual $k$-independence number $\alpha_k$. Improvements of the Inertia-type bound compared to the Plotkin-type bound and the Hamming-type bound are marked in bold.}
\label{tab: city block results}
\end{table}

We see that there are several instances where the Inertia-type bound performs better than both the Plotkin-type bound and the Hamming-type bound. In all of these instances the Inertia-type bound also performs as good as the Lovász theta number and is sharp. Moreover, there are many other instances where the Inertia-type bound is also sharp. 


\subsection{Projective metric}

The projective metric, introduced in \cite{Gabidulin1997MetricsSet} and recently investigated in \cite{SR2025}, is a general metric that depends on a specific choice of set. Many well-known metrics are instances of this metric for the right choice of set, like the Hamming metric and the sum-rank metric. 

\begin{definition} \label{def:projective}
Let $\mathcal{F}=\{F_1, \ldots, F_m\}$ be a set of one-dimensional subspaces of $\Fq^n$ such that $\text{span} \left( \bigcup_{i=1}^m F_i \right) = \Fq^n$. The \textit{projective $\mathcal{F}$-weight} of $\textbf{x} \in \Fq^n$ is defined as
\[
w_\mathcal{F}(\textbf{x}) := \min \left\{|I|: \textbf{x} \in \text{span} \left( \bigcup_{i \in I} F_i \right)\right\}.
\]
The \textit{projective $\mathcal{F}$-distance} between $\textbf{x},\textbf{y} \in \Fq^n$ is defined as $d_\mathcal{F}(\textbf{x},\textbf{y}) := w_\mathcal{F}(\textbf{x}-\textbf{y})$.
\end{definition}

From here on let $\mathcal{F}$ be such a set as in Definition \ref{def:projective}. We consider the discrete metric space $(\Fq^n, d_\mathcal{F})$ and apply the Eigenvalue Method to it. 
Define the \textit{projective $\mathcal{F}$-distance graph} $G_\mathcal{F}(\Fq^n)$ as the graph with vertex set $\Fq^n$ where vertices $\textbf{x},\textbf{y} \in \Fq^n$ are adjacent if $d_\mathcal{F}(\textbf{x},\textbf{y})=1$. We need to verify first if condition \ref{item: C1} is satisfied.

\begin{lemma} \label{lemma: proj geodesic dist}
The geodesic distance in $G_\mathcal{F}(\Fq^n)$ coincides with the projective $\mathcal{F}$-distance.
\end{lemma}

\begin{proof}
Let $\textbf{x},\textbf{y} \in \Fq^n$ with $d_\mathcal{F}(\textbf{x},\textbf{y})=d$. Then, by definition, there is a subset $\{i_1, \ldots,i_d\} \subseteq [m]$ of size $d$ such that 
\[
\textbf{x}-\textbf{y} \in \text{span} \left( \bigcup_{j=1}^d F_{i_j} \right).
\]
This means that for all $j \in \{1,\ldots, d\}$ there exists an $\textbf{f}_{i_j} \in F_{i_j}$ such that $\textbf{x}-\textbf{y} = \sum_{j=1}^d \textbf{f}_{i_j}$. Note that $(\textbf{x}, \textbf{x}-\textbf{f}_{i_1}, \textbf{x}-\textbf{f}_{i_1}-\textbf{f}_{i_2}, \ldots, \textbf{x}-\sum_{j=1}^d \textbf{f}_{i_j} = \textbf{y})$ is a path in $G_\mathcal{F}(\Fq^n)$ from $\textbf{x}$ to $\textbf{y}$ of length $d$. So the geodesic distance between $\textbf{x}$ and $\textbf{y}$ is at most $d$. By minimality of the cardinality of $\{i_1, \ldots, i_d\}$ the geodesic distance cannot be less than $d$. Hence the geodesic distance in $G_\mathcal{F}(\Fq^n)$ equals the projective $\mathcal{F}$-distance. 
\end{proof}

Since condition \ref{item: C1} is satisfied, the Eigenvalue Method is applicable. The next results show that the projective $\mathcal{F}$-distance graph has desired properties \ref{item: C2} and \ref{item: C3}.

\begin{lemma} \label{lemma: proj cayley}
The graph $G_\mathcal{F}(\Fq^n)$ is a Cayley graph over $\Fq^n$ with connecting set $S := \{\textbf{x} \in \Fq^n: w_\mathcal{F}(\textbf{x})=1\}$. 
\end{lemma}

\begin{proof}
Let $(\textbf{x},\textbf{y}) \in E(G_\mathcal{F}(\Fq^n))$. Then $w_\mathcal{F}(\textbf{x}-\textbf{y}) = d_\mathcal{F}(\textbf{x},\textbf{y})=1$, so $\textbf{x}-\textbf{y} \in S$ and $\textbf{x}=\textbf{y}+\textbf{s}$ for some $\textbf{s} \in S$. Now let $\textbf{x} \in \Fq^n$ and $\textbf{s} \in S$. Then $d_\mathcal{F}(\textbf{x},\textbf{x}+\textbf{s}) = w_\mathcal{F}(\textbf{s})=1$ since $\textbf{s} \in S$. So $\textbf{x}$ and $\textbf{x}+\textbf{s}$ are adjacent in $G_\mathcal{F}(\Fq^n)$.
\end{proof}

\begin{corollary} \label{cor: proj reg walk-r vertex}
The graph $G_\mathcal{F}(\Fq^n)$ is vertex-transitive, regular, and walk-regular.
\end{corollary}

\begin{proof}
Since Cayley graphs are vertex-transitive, the graph $G_\mathcal{F}(\Fq^n)$ is vertex-transitive. This immediately implies that $G_\mathcal{F}(\Fq^n)$ is regular and walk-regular.
\end{proof}

The last property to check is \ref{item: C4}. However, distance-regularity of the graph $G_\mathcal{F}(\Fq^n)$ depends on the choice of set $\mathcal{F}$. The set $\mathcal{F}_\text{H} = \{F_1, \ldots, F_n\}$ with $F_i = \text{span}(\textbf{e}_i)$ for $i=1, \ldots, n$, for instance, results in the Hamming distance, and the corresponding graph $G_{\mathcal{F}_\text{H}}(\Fq^n)$ is the Hamming graph, which is known to be distance-regular (see e.g.  \cite{Brouwer1989Distance-RegularGraphs}). On the other hand, the sum-rank metric, which we elaborate on below, results in a graph that is, in most instances, not distance-regular \cite[Proposition 12]{Abiad2023EigenvalueCodes}.

The sum-rank metric is an example of a projective metric. 
\begin{definition}\label{def:srk}
Let $t$ be a positive integer and let $\textbf{n}=(n_1,\ldots,n_t)$, $\textbf{m} = (m_1,\ldots,m_t)$ be ordered tuples of positive integers with $m_1 \geq m_2 \geq \cdots \geq m_t$, and $m_i\geq n_i$ for all $i\in [t]$. 
The \textit{sum-rank-metric space} is an $\mathbb{F}_q$-linear vector space $\mathbb{F}_q^{{\bf n}\times {\bf m}}$ defined as follows:
\[
\mathbb{F}_q^{{\bf n}\times {\bf m}}:=
\mathbb{F}_q^{n_1\times m_1}\times \cdots \times
\mathbb{F}_q^{n_t\times m_t}.
\]
The \textit{sum-rank} of an element $X=(X_1,\ldots, X_t)\in \mathbb{F}_q^{{\bf n}\times {\bf m}}$ is $\srk(X) := \sum_{i=1}^t \rk(X_i)$, where $\rk(X_i)$ denotes the rank of matrix $X_i$. The \textit{sum-rank distance} between $X,Y \in \mathbb{F}_q^{{\bf n}\times {\bf m}}$ is $\srk(X - Y)$.  
\end{definition}
Note that the sum-rank metric is indeed an instance of the projective metric: take the set $\mathcal{F}_\text{srk}$ containing all possible spans of a tuple of matrices, all equal to the zero matrix except for one which is a rank-one matrix. The sum-rank metric has been studied in the context of the Eigenvalue Method in \cite{Abiad2023EigenvalueCodes}. Abiad et al. establish various properties of the sum-rank-metric graph, which is the graph with vertex set $\Fq^{\textbf{n} \times \textbf{m}}$ where two vertices are adjacent if their sum-rank distance equals 1. Besides, new bounds on the maximum cardinality of sum-rank-metric codes are derived using the Ratio-type bound. These new bounds improve on the state of the art bounds for several choices of the parameters.

For specific choices of set $\mathcal{F}$, we want to be able to compare the results of the Eigenvalue Method to state of the art bounds. Depending on the set $\mathcal{F}$, bounds may exist for the specific metric arising in that case, like for the sum-rank metric. However, a bound for general codes in the projective metric also exists, namely a Singleton-type bound.
\begin{theorem}[\text{Singleton-type bound, \cite[Theorem 83]{SR2025}}] \label{th: singleton projective}
Let $\mathcal{C} \subseteq \Fq^n$ be a code of minimum projective $\mathcal{F}$-distance $d$. For $t \in \{0,\ldots, n\}$ define $\mu_\mathcal{F}(t)$ as the maximum cardinality of a subset $\mathcal{G} \subseteq \mathcal{F}$ such that:
\begin{itemize}
\item all $\textbf{f}_i \in \mathcal{G}$ are linearly independent over $\Fq$,
\item all $\textbf{v} \in \langle \mathcal{G} \rangle$ have $w_\mathcal{F}(\textbf{v}) \leq t$.
\end{itemize}
Then 
\[
|\mathcal{C}| \leq q^{n-\mu_\mathcal{F}(d-1)} \leq q^{n-d+1}.
\]
\end{theorem}

\subsection{Phase-rotation metric} \label{sec: phase rotation}
Another example of a projective metric is the phase-rotation metric. Note that although this is an instance of the previous metric, the phase-rotation metric is treated separately since we go into more depth with this metric.
 
The phase-rotation metric, which was introduced in \cite{Gabidulin1997CodesRotation}, is particularly suitable for decoding in a binary channel where errors are caused by phase inversions (where all zeros change to ones and all ones change to zeros), and random bit errors (where at random a zero changes to a one or a one changes to a zero). 

\begin{definition} \label{def:phase-rotation}
Let $\mathcal{F}_\text{pr}=\{F_1, \ldots, F_{n+1}\}$ be the set with $F_i = \text{span}(\textbf{e}_i)$ for $i=1, \ldots, n$ and $F_{n+1} = \text{span}(\textbf{1})$. The \textit{phase-rotation weight} $w_\text{pr}$ and the \textit{phase-rotation distance} $d_\text{pr}$ are defined as the projective $\mathcal{F}_\text{pr}$-weight and the projective $\mathcal{F}_\text{pr}$-distance from Definition \ref{def:projective}, respectively.
\end{definition}

\begin{example}
Consider the instance where $n=4, q=2$. Take $\textbf{x}=(0,0,0,0)$, $\textbf{y}=(1,0,0,1)$ and $\textbf{z}=(1,1,0,1)$, which are vectors in $\mathbb{F}_2^4$. Then the phase-rotation distance between $\textbf{x}$ and $\textbf{y}$ is 2 since the vectors differ in two coordinates and
\[
\textbf{x} - \textbf{y} = (0,0,0,0)-(1,0,0,1) = (1,0,0,1) = \textbf{e}_1 + \textbf{e}_4.
\]
The phase-rotation distance between $\textbf{x}$ and $\textbf{z}$ is also 2, even though the vectors differ in three coordinates, because 
\[
\textbf{x} - \textbf{z} = (0,0,0,0)-(1,1,0,1)=(1,1,0,1) = (1,1,1,1)+(0,0,1,0) = \textbf{1} + \textbf{e}_3.
\]
\end{example}

Now we apply the Eigenvalue Method to the discrete metric space $(\Fq^n, d_\text{pr})$ with the phase-rotation distance. Define the \textit{phase-rotation distance graph} $G_\text{pr}(\Fq^n)$ as the graph with vertex set $\Fq^n$ where vertices $\textbf{x},\textbf{y} \in \Fq^n$ are adjacent if $d_\text{pr}(\textbf{x},\textbf{y})=1$. Note that this is exactly the projective $\mathcal{F}$-distance graph for the specific set $\mathcal{F}=\mathcal{F}_\text{pr}$. Next we check the conditions of the Eigenvalue Method. Since the phase-rotation distance graph equals the projective $\mathcal{F}$-distance graph $G_\mathcal{F}(\Fq^n)$ for $\mathcal{F} = \mathcal{F}_\text{pr}$, condition \ref{item: C1} and properties \ref{item: C2} and \ref{item: C3} follow immediately.

\begin{corollary}
The geodesic distance in $G_\text{pr}(\Fq^n)$ coincides with the phase-rotation distance.
\end{corollary}

\begin{proof}
Since the phase-rotation distance is a projective distance for the specific set $\mathcal{F}_\text{pr}$ of Definition \ref{def:phase-rotation} and the phase-rotation distance graph $G_\text{pr}(\Fq^n)$ is defined accordingly, Lemma \ref{lemma: proj geodesic dist} directly implies that the geodesic distance in $G_\text{pr}(\Fq^n)$ coincides with the phase-rotation distance.
\end{proof}

\begin{corollary}
The graph $G_\text{pr}(\Fq^n)$ is a Cayley graph. Thus $G_\text{pr}(\Fq^n)$ is vertex-transitive, regular, and walk-regular. The degree of $G_\text{pr}(\Fq^n)$ is $q-1$ if $n=1$ and $(q-1)(n+1)$ if $n \geq 2$.
\end{corollary}

\begin{proof}
The graph properties follow immediately from the properties of the projective distance graph in Lemma \ref{lemma: proj cayley} and Corollary \ref{cor: proj reg walk-r vertex}. The degree of $G_\text{pr}(\Fq^n)$ is exactly the number of distinct nonzero vectors in $\cup_{i=1}^{n+1} F_i$. If $n=1$, $F_1=F_{n+1}$ and $F_1$ contains $q-1$ nonzero vectors, so the degree is $q-1$. If $n \geq 2$, all $F_i$ are disjoint and every $F_i$ contains $q-1$ nonzero vectors, so the degree is $(q-1)(n+1)$. 
\end{proof}

So condition \ref{item: C1} and properties \ref{item: C2} and \ref{item: C3} are met. Next we check property \ref{item: C4}. Distance-regularity does not follow immediately like the other properties of $G_\text{pr}(\Fq^n)$, but the following result gives a necessary and sufficient condition for distance-regularity of $G_\text{pr}(\Fq^n)$. 

\begin{proposition}
The graph $G_\text{pr}(\Fq^n)$ is distance-regular if and only if $n=1$, $n=2$ or $q=2$.
\end{proposition}

\begin{proof}
($\Leftarrow$) We prove the three cases, $n=1$, $n=2$ and $q=2$, separately. The case $n=1$ follows immediately: the graph for $n=1$ is a complete graph on $q$ vertices, which is distance-regular. 

When $q=2$, the phase-rotation distance graph equals the folded cube graph of dimension $n+1$, which is a distance-regular graph \cite[Section 9.2D]{Brouwer1989Distance-RegularGraphs}.

Consider the case $n=2$, $q \geq 3$. The diameter of $G_\text{pr}(\Fq^2)$ then equals 2. Observe that it suffices to show that for vertices $\textbf{u}$ and $\textbf{v}$ with $d(\textbf{u},\textbf{v})=i$ the values $b_i(\textbf{u},\textbf{v})$ and $c_i(\textbf{u},\textbf{v})$, the number of neighbors of $\textbf{u}$ at distance $i+1$, $i-1$ from $\textbf{v}$ respectively, do not depend on the choice of $\textbf{u}$ and $\textbf{v}$ for $i=0,1,2$. Since the phase-rotation distance graph is vertex-transitive, we may assume without loss of generality that $\textbf{v}=\textbf{0}$. Since $G_\text{pr}(\Fq^2)$ is $3(q-1)$-regular, we have $b_0 = 3(q-1)$ which does not depend on the choice of $\textbf{u}$. Also $c_0=0$, $c_1=1$, and $b_2=0$ by definition. 
Now consider $b_1(\textbf{u},\textbf{v})$. A vertex $\textbf{u}$ at distance $1$ from $\textbf{v}=\textbf{0}$ has the form $(a,0)$, $(0,a)$ or $(a,a)$ with $a \in \Fq^*$. The neighbors of $\textbf{v}$ at distance $2$ from $(a,0)$, $(0,a)$, $(a,a)$ with $a \in \Fq^*$ are
\[
\left\{(0,x), (y,y): x,y \in \Fq^*, x \neq -a,y \neq a \right\}, \quad \left\{(x,0), (y,y): x,y \in \Fq^*, x \neq -a,y \neq a \right\},
\]
\[
\left\{(0,x), (y,0): x,y \in \Fq^*, x,y \neq a \right\},
\]
respectively. All these three sets contain $2q-4$ distinct vertices so $b_1 = 2q-4$, which is independent of the choice of $\textbf{u}$. 
Now consider $c_2(\textbf{u},\textbf{v})$. A vertex $\textbf{u}$ at distance $2$ from $\textbf{v}=\textbf{0}$ has the form $(a,b)$ with $a,b \in \Fq^*, a \neq b$. The neighbors of $\textbf{v}$ that are also neighbors $(a,b)$ with $a,b \in \Fq^*, a \neq b$ are the vertices
\[
\{(0,b), (a,0), (a,a), (b,b), (0, b-a), (a-b,0)\}.
\]
These are six distinct vertices, independent of the choice of $a,b \in \Fq^*$ such that $a \neq b$, so independent of the choice of $\textbf{u}$, and hence $c_2=6$. This proves that $G_\text{pr}(\Fq^2)$ for $q \geq 3$ is distance-regular.

($\Rightarrow$) Now we show that $n=1$, $n=2$ and $q=2$ are the only cases where the phase-rotation distance graph is distance-regular. Let $n \geq 3$ and $q \geq 3$. 
Define $r := \lceil \frac{n}{2} \rceil$. Let 
\[
\textbf{x} := (\underbrace{1,\ldots, 1}_{r-1 \text{ times}}, a, 0, \ldots,0), \qquad \textbf{y} := (\underbrace{1, \ldots, 1}_{r \text{ times}},0, \ldots, 0),
\]
for some fixed $a \in \Fq^*, a \neq 1$. Note that $d_\text{pr}(\textbf{x},\textbf{0}) = d_\text{pr}(\textbf{y},\textbf{0}) = r$. 

When $n$ is odd, consider $c_r$. For $\textbf{x}$ and the zero vector we find
\[
c_r(\textbf{x},\textbf{0}) = p_{1, r-1}^r(\textbf{x},\textbf{0}) = r,
\]
since the only neighbors of $\textbf{x}$ at distance $r-1$ from the zero vector are the vectors where one of the nonzero coordinates of $\textbf{x}$ is replaced with a zero. For $\textbf{y}$ and the zero vector we find
\[
c_r(\textbf{y},\textbf{0}) = p_{1, r-1}^r(\textbf{y},\textbf{0}) \geq r+1,
\]
since $\textbf{y}$ has $r$ neighbors at distance $r-1$ from the zero vector similarly as $\textbf{x}$ and the zero vector, but $\textbf{y}$ also has the vector starting with $r+1$ ones and then $r-2$ zeros, which is at distance $r-1$ from the zero vector, as a neighbor. So the number $c_r$ for $n$ odd depends on the choice of vertices. 

When $n$ is even we consider $a_r$. For $\textbf{x}$ and the zero vector we get:
\[
a_r(\textbf{x},\textbf{0}) = p_{1,r}^r(\textbf{x},\textbf{0}) = (q-2)r,
\]
since the only neighbors of $\textbf{x}$ at distance $r$ from the zero vector are the vectors where one of the $r$ nonzero coordinates of $\textbf{x}$ is replaced by another nonzero element of $\Fq$. For $\textbf{y}$ and the zero vector we have
\[
a_r(\textbf{y},\textbf{0})=p_{1,r}^r(\textbf{y},\textbf{0}) \geq (q-2)r+1, 
\]
since $\textbf{y}$ has $(q-2)r$ neighbors at distance $r$ from the zero vector similarly as $\textbf{x}$ and the zero vector, but $\textbf{y}$ also has the vector starting with $r+1$ ones and then $r-1$ zeros, which is at distance $r$ from the zero vector, as neighbor.
So the value of $a_r$ depends on the choice of vertices for $n$ even. Hence $G_\text{pr}(\Fq^n)$ is not distance-regular when $n \geq 3, q \geq 3$, which proves the result.
\end{proof}

The latter result shows that property \ref{item: C4} is met when $n \geq 3$ and $q \geq 3$. Moreover, since the graph $G_\text{pr}(\Fq^n)$ has the desired properties \ref{item: C2} and \ref{item: C3}, both the Inertia-type bound and the Ratio-type bound can be applied to this graph. To do so, we first need to determine the adjacency eigenvalues of the phase-rotation distance graph. 

\begin{proposition} \label{prop: eigval pr graph}
The adjacency eigenvalues of $G_\text{pr}(\Fq^n)$ for $n \geq 2$ are 
\[
\lambda_\textbf{r} = \left(\sum_{l=1}^n q \mathbbm{1} \{r_l = 0\}\right) + q \mathbbm{1} \left\{\sum_{l=1}^n r_l \equiv 0 \mod q \right\} - n - 1
\]
for every tuple $\textbf{r}= (r_1, \ldots, r_n) \in [\![q-1]\!]^n$.
\end{proposition}
Note that $\mathbbm{1}$ denotes the indicator function.

\begin{remark}
The latter result does not give the eigenvalues of $G_\text{pr}(\Fq^n)$ for $n=1$. However, in this case, the phase-rotation distance graph is equivalent to the complete graph on $q$ vertices $K_q$. The adjacency eigenvalues of $G_\text{pr}(\Fq)$ are thus $q-1$ and $-1$ with respective multiplicities $1$ and $q-1$ (see for instance \cite{Cvetkovic2009ApplicationsLiterature}).
\end{remark}

For the proof of Proposition \ref{prop: eigval pr graph} we use characters of groups. So we first present some background on characters and how they can be used to determine the eigenvalues of a Cayley graph. For more details about the latter, we refer the reader to \cite{lovasz1975spectra}.

\begin{definition}
Let $G$ be a group. A function $\chi:G \mapsto \mathbb{C}$ is a \textit{character} of $G$ if $\chi$ is a group homomorphism from $G$ to $\mathbb{C} \backslash \{0\}$ and $|\chi(g)| = 1$ for every $g \in G$.
\end{definition}

\begin{example} \label{ex: characters Z/mZ}
The characters of $\mathbb{Z}/m \mathbb{Z}$ are $\chi_r(x) := (\zeta_m)^{r x}$ for $r=0,\ldots, m-1$, where $\zeta_m:= \exp(\tfrac{2 \pi i}{m})$ denotes the $m$-th root of unity \cite{lovasz1975spectra}.
\end{example}

Example \ref{ex: characters Z/mZ} gives the characters of cyclic groups. Note that any finite abelian group $G$ is isomorphic to a product of cyclic groups, i.e. $G \cong \mathbb{Z}/m_1 \mathbb{Z} \times \cdots \times \mathbb{Z}/m_l \mathbb{Z}$. It turns out that the characters of a Cartesian product of groups are related to the characters of the individual groups. 

\begin{lemma}
\label{lem: characters cart prod}
Let $G,H$ be finite abelian groups with characters $\chi_{G,i}, \chi_{H,j}$ respectively. The characters of $G \times H$, the Cartesian product of $G$ and $H$, are $\chi_{i,j}((g,h)) := \chi_{G,i}(g) \cdot \chi_{H,j}(h)$. 
\end{lemma}

\begin{proof}
Let $*, *'$ denote the operation in $G$, respectively $H$. Let $(g_1,h_1), (g_2,h_2) \in G \times H$. We want to prove that $\chi_{i,j}((g_1,h_1)(* \times *')(g_2,h_2)) = \chi_{i,j}((g_1,g_2)) \chi_{i,j}((g_2,h_2))$ since this is a sufficient condition for $\chi_{i,j}$ to be a group homomorphism. Observe:
\[
\chi_{i,j}((g_1,h_1)(* \times *')(g_2,h_2)) = \chi_{i,j}((g_1*g_2,h_1*'h_2)) = \chi_{G,i}(g_1*g_2) \chi_{H,j}(h_1*'h_2)
\]
and
\[
\chi_{i,j}((g_1,g_2)) \chi_{i,j}((g_2,h_2)) = \chi_{G,i}(g_1)\chi_{H,j}(g_2)\chi_{G,i}(h_1)\chi_{H,j}(h_2) = \chi_{G,i}(g_1*g_2) \chi_{H,j}(h_1*'h_2),
\]
since $\chi_{G,i}$ and $\chi_{H,j}$ are group homomorphisms of $G, H$, respectively. So $\chi_{i,j}$ is a group homomorphism of $G \times H$. Moreover, for any $(g,h) \in G \times H$, we have 
\[
|\chi_{i,j}((g,h))| = |\chi_{G,i}(g)| \cdot |\chi_{H,j}(h)| = 1,
\]
since $\chi_{G,i}$ and $\chi_{H,j}$ are characters of $G,H$ respectively. Hence $\chi_{i,j}$ is a character of $G \times H$.
\end{proof}

With Lemma \ref{lem: characters cart prod} and Example \ref{ex: characters Z/mZ}, the characters of finite abelian groups are now completely defined. The next result tells us how characters determine the adjacency eigenvalues of a Cayley graph, which is the final bit of information needed for the proof of Proposition \ref{prop: eigval pr graph}.

\begin{lemma}[\text{\cite{lovasz1975spectra} 
}] \label{lemma: characters cayley graph}
Let $G$ be a finite abelian group, let $\chi_i$ for be the characters of $G$, and let $S \subseteq G$ be a symmetric set. The adjacency eigenvalues of the Cayley graph over group $G$ with connecting set $S$ are given by
\[
\lambda_i = \sum_{s \in S} \chi_i(s),
\]
for $i = 0, \ldots, |G|-1$.
\end{lemma}

\begin{proof}[Proof of Proposition \ref{prop: eigval pr graph}.]
The graph $G_\text{pr}(\Fq^n)$ is a Cayley graph over $\Fq^n$ with connecting set $S := \{c\textbf{x}: c \in \Fq^*, \textbf{x} \in \{\textbf{e}_1, \ldots, \textbf{e}_n, \textbf{1}\} \}$. First we determine the characters of $\Fq$. The field $\Fq$ seen as a group is isomorphic to $(\mathbb{Z}/p\mathbb{Z})^k$ for some prime $p$ such that $q=p^k$. Since the characters of $\mathbb{Z}/p\mathbb{Z}$ are known to be $\chi_r(x) = (\zeta_p)^{rx}$ for $r=0, \ldots, p-1$ where $\zeta_p = \exp \left( {2 \pi i}/{p} \right)$ (see Example \ref{ex: characters Z/mZ}), Lemma \ref{lem: characters cart prod} tells us that the characters of $\Fq$ are 
\[
\chi_\textbf{r}(\textbf{x}) = \prod_{j=1}^k (\zeta_p)^{r_jx_j} = (\zeta_p)^{\sum_{j=1}^k r_jx_j},
\]
for $\textbf{r} \in [\![p-1]\!]^k$, where $\textbf{x} = (x_1, \ldots, x_k) \in (\mathbb{Z}/p\mathbb{Z})^k \cong \Fq$. The multiplication $r_jx_j$ is taken modulo $p$; this abuse of notation is used more often in this proof. 
Now we can determine the characters of $\Fq^n$, again using Lemma \ref{lem: characters cart prod}:
\begin{equation} \label{eq: characters Fqn}
\chi_\textbf{r}((\textbf{x}_1, \ldots, \textbf{x}_n)) = \prod_{l=1}^n \chi_{\textbf{r}_l}(\textbf{x}_l) = \prod_{l=1}^n (\zeta_p)^{\sum_{j=1}^k r_{l_j}x_{l_j}} = (\zeta_p)^{\sum_{l=1}^n \sum_{j=1}^k r_{l_j}x_{l_j}},
\end{equation}
for $\textbf{r} = (\textbf{r}_1, \ldots, \textbf{r}_n)$ with $\textbf{r}_l = (r_{l_1}, \ldots, r_{l_k}) \in [\![p-1]\!]^k$, $l=1, \ldots, n$, where $\textbf{x}_l \in (\mathbb{Z}/p\mathbb{Z})^k \cong \Fq$, $l=1, \ldots, n$. 
By Lemma \ref{lemma: characters cayley graph} the adjacency eigenvalues of $G_\text{pr}(\Fq^n)$ are then
\[
\lambda_\textbf{r} = \sum_{\textbf{s} \in S} \chi_\textbf{r}(\textbf{s}),
\]
for tuples $\textbf{r} \in ([\![p-1]\!]^k)^n$, where every $\textbf{s} \in S$ is viewed as an element of $([\![p-1]\!]^k)^n$.

Now we simplify this expression. Let $\textbf{s} \in S$, then $\textbf{s} = c \textbf{x}$ for some $c \in \Fq^*$ and $\textbf{x} \in \{\textbf{e}_1, \ldots, \textbf{e}_n, \textbf{1}\}$. If $\textbf{s}$ is viewed as an element of $([\![p-1]\!]^k)^n$, then $\textbf{s} = (\textbf{c}, \textbf{0}, \ldots, \textbf{0}), \ldots, (\textbf{0}, \ldots, \textbf{0}, \textbf{c})$ or $(\textbf{c}, \ldots, \textbf{c})$ for some $\textbf{c} \in [\![p-1]\!]^k$, $\textbf{c} \neq \textbf{0}$. So for a fixed $\textbf{r} \in ([\![p-1]\!]^k)^n$ we get:
\[
\lambda_\textbf{r} = \sum_{\textbf{c} \in [\![p-1]\!]^k, \textbf{c} \neq \textbf{0}} \chi_\textbf{r}((\textbf{c}, \textbf{0}, \ldots, \textbf{0})) + \cdots + \chi_\textbf{r}((\textbf{0}, \ldots, \textbf{0}, \textbf{c})) + \chi_\textbf{r}((\textbf{c}, \ldots, \textbf{c})).
\]
Since $\chi_{r_l}(\textbf{0)} = 1$ for $l=1, \ldots, n$, this simplifies to
\[
\lambda_\textbf{r} = \sum_{\textbf{c} \in [\![p-1]\!]^k, \textbf{c} \neq \textbf{0}} \chi_{\textbf{r}_1}(\textbf{c}) + \cdots + \chi_{\textbf{r}_n}(\textbf{c}) + \prod_{l=1}^n \chi_{\textbf{r}_l}(\textbf{c}).
\]
Using the expression for the characters from Equation \eqref{eq: characters Fqn} and letting the sum also run over $\textbf{c}=\textbf{0}$ gives:
\begin{equation} \label{eq: lambda characters with sums}
\lambda_\textbf{r} = \sum_{\textbf{c} \in [\![p-1]\!]^k} \sum_{l=1}^n (\zeta_p)^{\sum_{j=1}^k r_{l_j}c_j} + (\zeta_p)^{\sum_{l=1}^n \sum_{j=1}^k r_{l_j} c_j} - n - 1.
\end{equation}

We know that $1+ \zeta_m + \cdots + (\zeta_m)^{m-1} = 0$ for any $m$-th root of unity $\zeta_m \neq 1$. Since $\sum_{j=1}^k r_{l_j}c_j \mod p$ attains every value of $\{0, \ldots, p-1\}$ equally often when $\textbf{r}_l \neq \textbf{0}$ for $\textbf{c} \in [\![p-1]\!]^k$, we get 
\[
\sum_{\textbf{c} \in [\![p-1]\!]^k} (\zeta_p)^{\sum_{j=1}^k r_{l_j}c_j} = 0
\]
for $l=1, \ldots, n$ if $\textbf{r}_l \neq \textbf{0}$. If $\textbf{r}_l = \textbf{0}$, then 
\[
\sum_{\textbf{c} \in [\![p-1]\!]^k} (\zeta_p)^{\sum_{j=1}^k r_{l_j}c_j} = \sum_{\textbf{c} \in [\![p-1]\!]^k} 1 = p^k = q.
\]
Also $\sum_{l=1}^n \sum_{j=1}^k r_{l_j} c_j \mod p$ attains every value in $\{0,\ldots, p-1\}$ equally often when $\sum_{l=1}^n \textbf{r}_l = (\sum_{l=1}^n r_{l_1}, \ldots, \sum_{l=1}^n r_{l_k}) \not \equiv \textbf{0} \mod p$ for $\textbf{c} \in [\![p-1]\!]^k$. So 
\[
\sum_{\textbf{c} \in [\![p-1]\!]^k} (\zeta_p)^{\sum_{l=1}^n \sum_{j=1}^k r_{l_j} c_j} = 0
\]
if $\sum_{l=1}^n \textbf{r}_l \not \equiv \textbf{0} \mod p$. If  $\sum_{l=1}^n \textbf{r}_l \equiv \textbf{0} \mod p$, then
\[
\sum_{\textbf{c} \in [\![p-1]\!]^k} (\zeta_p)^{\sum_{l=1}^n \sum_{j=1}^k r_{l_j} c_j} = \sum_{\textbf{c} \in [\![p-1]\!]^k} (\zeta_p)^{\sum_{j=1}^k 0 \cdot c_j} = \sum_{\textbf{c} \in [\![p-1]\!]^k} 1 = p^k = q.
\]
Combining these four observations with Equation \eqref{eq: lambda characters with sums} gives the following formula for the eigenvalues:
\[
\lambda_\textbf{r} = \left( \sum_{l=1}^n q \mathbbm{1}\{\textbf{r}_l = \textbf{0}\} \right) + q \mathbbm{1} \left\{\sum_{l=1}^n \textbf{r}_l \equiv \textbf{0} \mod p \right\} - n -1,
\]
for $\textbf{r}=(\textbf{r}_1, \ldots, \textbf{r}_n) \in ([\![p-1]\!]^k)^n$. 

For the last step of this proof, we note that $(a_1, \ldots, a_k) \in [\![p-1]\!]^k$ can be related to $a \in[\![q-1]\!]$, where $q=p^k$, by $a = \sum_{j=1}^k a_j p^{j-1}$. Then the conditions $\textbf{r}_l = \textbf{0}$ and $\sum_{l=1}^n \textbf{r}_l \equiv \textbf{0} \mod p$ are equivalent to the conditions $r_l = 0$ and $\sum_{l=1}^n r_l \equiv 0 \mod q$, respectively. Using this observation, the eigenvalues of $G_\text{pr}(\Fq^n)$ for $n \geq 2$ equal
\[
\lambda_\textbf{r} = \left( \sum_{l=1}^n q \mathbbm{1} \{r_l = 0\} \right) + q \mathbbm{1} \left\{ \sum_{l=1}^n r_l \equiv 0 \mod q \right\} - n - 1,
\]
for tuples $\textbf{r}=(r_1, \ldots, r_n) \in [\![q-1]\!]^n$, which is what we wanted to prove.
\end{proof}

Now we can derive the distinct eigenvalues of the phase-rotation distance graph for $n \geq 2$; these distinct eigenvalues follow directly from Proposition \ref{prop: eigval pr graph}.

\begin{corollary}
The distinct adjacency eigenvalues of $G_\text{pr}(\Fq^n)$ for $n \geq 2$ are
\[
\begin{cases}
2i - n -1, \quad \text{for } i=1,3, \ldots, n-1, n+1 &\text{if } q=2, n \text{ even}, \\
2i-n-1, \quad \text{for } i=0, 2, \ldots, n-1, n+1 &\text{if } q=2, n \text{ odd}, \\
iq - n - 1, \quad \text{for } i=0, 1, \ldots, n-1, n+1 &\text{if } q \geq 3. \\
\end{cases}
\]
\end{corollary}

The expressions for the (distinct) eigenvalues of the phase-rotation adjacency graph can be used in the Inertia-type bound and the Ratio-type bound to derive new bounds on the cardinality of phase-rotation codes. Then we compare these new bounds to a state of the art bound, namely a Singleton-type bound for phase-rotation codes.

\begin{theorem}[\text{Singleton-type bound, \cite{SR2025}}]
\label{th: singleton pr}
Let $\mathcal{C} \subseteq \Fq^n$ be a code of minimum phase-rotation distance $d$. Then 
\[
|\mathcal{C}| \leq \begin{cases}
    q^{n-d+1} &\text{if } d<1+\lceil n-\tfrac{n}{q} \rceil, \\
    1 &\text{otherwise}.
\end{cases}
\]
\end{theorem}
This bound follows from the Singleton-type bound for the projective metric from Theorem \ref{th: singleton projective} by using the result of \cite[Example 81]{SR2025}, which states that 
\[
w_{\mathcal{F}_\text{pr}}(t) = \begin{cases}
    t &\text{if } t< \lceil n -\tfrac{n}{q} \rceil, \\
    n &\text{otherwise},
\end{cases}
\]
for the set $\mathcal{F}_\text{pr}$ as defined in Definition \ref{def:phase-rotation}.

We start by considering the Ratio-type bound on the $k$-independence number for $k=1,2,3$, since there are explicit expressions for this bound that are independent of a choice of polynomial $p \in \R_k[x]$ (see \cite[Theorem 3.2]{Haemers1995InterlacingGraphs}, 
Theorem \ref{th: alpha_2}, and Theorem \ref{th: alpha3}, respectively). First consider the ratio bound on the independence number $\alpha$, i.e. the Ratio-type bound on the $k$-independence number $\alpha_k$ for $k=1$. Applied to the phase-rotation distance graph $G_\text{pr}(\Fq^n)$, this ratio bound gives the following upper bound on the independence number $\alpha(G_\text{pr}(\Fq^n))$.

\begin{theorem}\label{th: alpha pr graph}
Let $n \geq 2$. Then
\[
\alpha(G_\text{pr}(\Fq^n)) \leq \begin{cases}
2^{n-1} \frac{n-1}{n} & \text{if } q=2, n \text{ even}, \\
q^{n-1} &\text{if } q=2, n \text{ odd or } q \geq 3.
\end{cases}
\]
\end{theorem}

\begin{proof}
The largest eigenvalue is $\lambda_1 = (n+1)(q-1)$ for all $q,n$, while the smallest eigenvalue is $\lambda_{q^n} = 1-n$ if $q=2$, $n$ even and $\lambda_{q^n} = -n-1$ otherwise. The ratio bound from \cite[Theorem 3.2]{Haemers1995InterlacingGraphs} 
is applicable, so for $q=2$, $n$ even we get
\[
\alpha(G_\text{pr}(\mathbb{F}_2^n)) \leq 2^n \frac{-(1-n)}{(n+1)-(1-n)} = 2^{n-1} \frac{n-1}{n}.
\]
For $q=2$, $n$ odd or $q \geq 3$, we obtain
\[
\alpha(G_\text{pr}(\Fq^n)) \leq q^n \frac{-(-n-1)}{(n+1)(q-1)-(-n-1)} = q^n \frac{n+1}{q(n+1)} = q^{n-1}.
\]
\end{proof}

Let $A_q^\text{pr}(n,d)$ denote the maximum cardinality of code in $\Fq^n$ with minimum phase-rotation distance $d$. The upper bounds from Theorem \ref{th: alpha pr graph} can be translated to upper bounds on $A_q^\text{pr}(n,d)$ via Lemma \ref{lemma: geo dist then max card alphak}. 
\begin{corollary} \label{cor: pr alpha1}
The cardinality of phase-rotation codes in $\Fq^n$ of minimum distance 2 with $n \geq 2$ is upper bounded by:
\begin{equation} \label{eq: pr alpha1}
A_q^\text{pr}(n,2) \leq \begin{cases}
2^{n-1} \frac{n-1}{n} & \text{if } q=2, n \text{ even}, \\
q^{n-1} &\text{if } q=2, n \text{ odd or } q \geq 3.
\end{cases}
\end{equation}
\end{corollary}
Next we compare these upper bounds from Corollary \ref{cor: pr alpha1} to the Singleton-type upper bound from Theorem \ref{th: singleton pr}. 

\begin{proposition}
Let $n \geq 2$. The upper bounds on $A_q^\text{pr}(n,2)$ in Equation \eqref{eq: pr alpha1}, which are a consequence of the ratio bound, are no worse than the upper bound from the Singleton-type bound of Theorem \ref{th: singleton pr}.
\end{proposition}

\begin{proof}
The upper bound from the Singleton-type bound for $d=2$ is $q^{n-1}$ if $2<1+\lceil n - \tfrac{n}{q} \rceil$. This upper bound applies exactly if $n-\tfrac{n}{q} >1 \Leftrightarrow n>1+\tfrac{1}{q-1}$. If $q=2$, then we need $n>2$, and if $q\geq 3$, then $n \geq 2$ suffices to satisfy the condition $n>1+\tfrac{1}{q-1}$. In these cases, we can immediately see that both upper bounds from Equation \eqref{eq: pr alpha1} are at most $q^{n-1}$. 

In the other case, namely $q=2$, $n=2$, the Singleton-type upper bound for $d=2$ is $1$. But in this case our bound gives an upper bound of $2^{2-1} \cdot \tfrac{2-1}{2} = 1$. So our bounds of Equation \eqref{eq: pr alpha1} are no worse than than the Singleton-type bound.
\end{proof}

So the Ratio-type bound on the $k$-independence number for $k=1$ gives a bound on the maximum cardinality of phase-rotation codes that is at least as good as the Singleton-type bound. Next we consider the Ratio-type bound on the $2$-independence number $\alpha_2$. The cases $q=2$ and $q \geq 3$ are treated separately since the expression for the distinct eigenvalues of $G_\text{pr}(\Fq^n)$ for these cases is sufficiently different. First consider the case $q=2$ (and $k=2$).

\begin{theorem} \label{alpha2 q2 pr graph}
Let $n \geq 3$. Then
\[
\alpha_2(G_\text{pr}(\mathbb{F}_2^n)) \leq \begin{cases}
2^n \frac{n-2}{n(n+4)} & \text{if } n \equiv 0 \mod 4, \\
2^n \frac{n-3}{(n+3)(n-1)} & \text{if } n \equiv 1 \mod 4, \\
2^n \frac{1}{n+2} & \text{if } n \equiv 2 \mod 4, \\
2^n \frac{1}{n+5} & \text{if } n \equiv 3 \mod 4 .
\end{cases}
\]
\end{theorem}

\begin{proof}
Since $n \geq 3$, $G_\text{pr}(\mathbb{F}_2^n)$ has at least three distinct eigenvalues, so Theorem \ref{th: alpha_2} is applicable. The largest eigenvalue which is at most $-1$ satisfies:
\[
2i-n-1 \leq -1 \Leftrightarrow 2i \leq n \Leftrightarrow i \leq \frac{n}{2}.
\]

We start with the case where $n$ is even, or $n \equiv 0,2 \mod 4$. 
Since $i$ has to be odd for $2i-n-1$ to be an eigenvalue when $n$ is even, we get $i = \frac{n}{2}$ if $n \equiv 2 \mod 4$ and $i = \frac{n}{2}-1$ if $n \equiv 0 \mod 4$. If $n \equiv 2 \mod 4$, we have $\theta_i = -1, \theta_{i-1} = 3, \theta_0 = n+1$. Then
\[
\alpha_2(G_\text{pr}(\mathbb{F}_2^n)) \leq 2^n \frac{n+1-3}{(n+2)(n-2)} = \frac{2^n}{n+2}.
\]
If $n \equiv 0 \mod 4$, we have $\theta_i = -3, \theta_{i-1} = 1, \theta_0 = n+1$. Then
\[
\alpha_2(G_\text{pr}(\mathbb{F}_2^n)) \leq 2^n \frac{n+1-3}{(n+4)n} = 2^n \frac{n-2}{n(n+4)}.
\]

Next we deal with the case where $n$ is odd, or $n \equiv 1,3 \mod 4$. Since $i$ has to be even for $2i-n-1$ to be an eigenvalue when $n$ is odd, we get $i = \frac{n-1}{2}$ if $n \equiv 1 \mod 4$ and $i = \frac{n-1}{2}-1$ if $n \equiv 3 \mod 4$. If $n \equiv 1 \mod 4$, we have $\theta_i = -2, \theta_{i-1} = 2, \theta_0 = n+1$. Then
\[
\alpha_2(G_\text{pr}(\mathbb{F}_2^n)) \leq 2^n \frac{n+1-4}{(n+3)(n-1)} = 2^n \frac{n-3}{(n+3)(n-1)}.
\]
If $n \equiv 3 \mod 4$, we have $\theta_i = -4, \theta_{i-1} = 0, \theta_0 = n+1$. Then
\begin{align*}
&\alpha_2(G_\text{pr}(\mathbb{F}_2^n)) \leq 2^n \frac{n+1}{(n+5)(n+1)} = \frac{2^n}{n+5}.
\qedhere\end{align*}
\end{proof}

The upper bounds from Theorem \ref{alpha2 q2 pr graph} can be translated to upper bounds on $A^\text{pr}_2(n,3)$ via Lemma \ref{lemma: geo dist then max card alphak}. 
\begin{corollary} \label{cor: alpha2 q2}
The maximum cardinality of phase-rotation codes in $\mathbb{F}_2^n$ of minimum distance 3 with $n \geq 3$ is upper bounded by
\begin{equation} \label{eq:pr alpha2 q2}
A_2^\text{pr}(n,3) \leq \begin{cases}
2^n \frac{n-2}{n(n+4)} & \text{if } n \equiv 0 \mod 4, \\
2^n \frac{n-3}{(n+3)(n-1)} & \text{if } n \equiv 1 \mod 4, \\
2^n \frac{1}{n+2} & \text{if } n \equiv 2 \mod 4, \\
2^n \frac{1}{n+5} & \text{if } n \equiv 3 \mod 4. 
\end{cases}
\end{equation}
\end{corollary}
Now we compare these upper bounds from Corollary \ref{cor: alpha2 q2} to the Singleton-type upper bound from Theorem \ref{th: singleton pr}.

\begin{proposition}
Let $n \geq 3$. The upper bounds on $A_2^\text{pr}(n,3)$ in Equation \eqref{eq:pr alpha2 q2}, which resulted from the Ratio-type bound, are no worse than the upper bound from the Singleton-type bound of Theorem \ref{th: singleton pr}. 
\end{proposition}

\begin{proof}
The upper bound of the Singleton-type bound for $d=3$ and $q=2$ is $2^{n-2}$ if $3<1+ \lceil n-\tfrac{n}{2} \rceil$, which is exactly if $\tfrac{n}{2} > 2 \Leftrightarrow n>4$. If this is the case, then we can compare the bounds and determine when the bounds from Equation \eqref{eq:pr alpha2 q2} are smaller than or equal to $2^{n-2}$. For $n \equiv 0 \mod 4$ we have:
\[
2^n \frac{n-2}{n(n+4)} \leq 2^{n-2} \Leftrightarrow 2^2(n-2) \leq n(n+4) \Leftrightarrow n^2 \geq -8,
\]
which trivially holds true. For $n \equiv 1 \mod 4$:
\[
2^n \frac{n-3}{(n+3)(n-1)} \leq 2^{n-2} \Leftrightarrow 2^2(n-3) \leq (n+3)(n-1) \Leftrightarrow n^2-2n+9 = (n-1)^2+8 \geq 0.
\]
Also this holds true. For $n \equiv 2 \mod 4$:
\[
2^n \frac{1}{n+2} \leq 2^{n-2} \Leftrightarrow 2^2 \leq n+2 \Leftrightarrow n \geq 2,
\]
which is true by the assumption on $n$. Lastly, for $n \equiv 3 \mod 4$ we get:
\[
2^n \frac{1}{n+5} \leq 2^{n-2} \Leftrightarrow 2^2 \leq n+5 \Leftrightarrow n \geq -1,
\]
which also holds by the assumption on $n$.

Next we consider the case that $n=3,4$. Then the Singleton-type bound gives an upper bound of $1$, while our bounds give values of $2^3 \cdot \tfrac{1}{8} = 1$ and $2^4 \cdot \tfrac{2}{32}=1$ for $n=3,4$ respectively. Hence the upper bounds on $A_2^\text{pr}(n,3)$ from Equation \eqref{eq:pr alpha2 q2} are no worse than the Singleton-type upper bound from Theorem \ref{th: singleton pr}.
\end{proof}

Hence the Ratio-type bound gives upper bounds on the cardinality of phase-rotation codes that are at least as good as the Singleton-type bound for $k=2$ and $q=2$. Next we consider the case $k=2$ and $q \geq 3$.

\begin{theorem} \label{th: alpha2 q3 pr graph}
Let $n \geq 2$ and $q \geq 3$. Then 
\[
\alpha_2(G_\text{pr}(\Fq^n)) \leq q^{n-2} \frac{n(n+1)+\lfloor \tfrac{n}{q} \rfloor q \big( -2-2n+q+ \lfloor \tfrac{n}{q} \rfloor q \big) }{ \big( n- \lfloor \tfrac{n}{q} \rfloor \big) \big( n+1-\lfloor \tfrac{n}{q} \rfloor \big)}.
\]
\end{theorem}

\begin{proof}
The distinct eigenvalues of $G_\text{pr}(\Fq^n)$ for $n\geq 2$ and $q \geq 3$ are $iq-n-1$ for $i=0,1\ldots, n-1,n+1$. Since $n \geq 2$, we have at least 3 distinct eigenvalues, so Theorem \ref{th: alpha_2} is applicable. First we determine the largest eigenvalue which is at most $-1$:
\[
iq-n-1 \leq -1 \Leftrightarrow iq \leq n \Leftrightarrow i \leq \frac{n}{q}.
\]
Taking $i = \lfloor \frac{n}{q} \rfloor$ gives this eigenvalue. Note that $0 \leq \lfloor \frac{n}{q} \rfloor \leq \frac{n}{3} \leq n-1$, so $i = \lfloor \frac{n}{q} \rfloor$ indeed gives an eigenvalue. Using the notation of Theorem \ref{th: alpha_2}, we have
\[
\theta_0 = (n+1)(q-1), \quad \theta_{i-1} = \left( \lfloor \tfrac{n}{q} \rfloor +1 \right)q-n-1, \quad \theta_i = \lfloor \tfrac{n}{q} \rfloor q -n-1.
\]
Then we obtain the following upper bound for $\alpha_2(G_\text{pr}(\Fq^n))$:
\[
q^n \frac{(n+1)(q-1) + \left( \lfloor \tfrac{n}{q} \rfloor q -n-1 \right) \left( ( \lfloor \tfrac{n}{q} \rfloor +1 )q-n-1 \right)}{\left( (n+1)(q-1) - (\lfloor \tfrac{n}{q} \rfloor q -n-1) \right) \left( (n+1)(q-1) - \big( ( \lfloor \tfrac{n}{q} \rfloor +1 ) q-n-1 \big) \right)}
\]
\[
= q^{n-2} \frac{n(n+1)+ \lfloor \tfrac{n}{q} \rfloor q \big( -2-2n+q+ \lfloor \tfrac{n}{q} \rfloor q \big)}{\big( n-\lfloor \tfrac{n}{q} \rfloor \big) \big( n+1-\lfloor \tfrac{n}{q} \rfloor \big)}.
\]
\end{proof}

This upper bound from Theorem \ref{th: alpha2 q3 pr graph} can be translated to an upper bound on $A_q^\text{pr}(n,3)$ via Lemma \ref{lemma: geo dist then max card alphak}.
\begin{corollary} \label{cor: pr alpha2}
The cardinality of phase-rotation codes of minimum distance $3$ with $q \geq 3$ and $n \geq 2$ is upper bounded by
\begin{equation} \label{eq:pr alpha2}
A_q^\text{pr}(n,3) \leq q^{n-2} \frac{n(n+1)+ \lfloor \tfrac{n}{q} \rfloor q \big( -2-2n+q+ \lfloor \tfrac{n}{q} \rfloor q \big)}{\big( n-\lfloor \tfrac{n}{q} \rfloor \big) \big( n+1-\lfloor \tfrac{n}{q} \rfloor \big)}.
\end{equation}
\end{corollary}
A comparison of this upper bound from Corollary \ref{cor: pr alpha2} with the Singleton-type bound from Theorem \ref{th: singleton pr} gives the following result.

\begin{proposition}
Let $n \geq 2$, $q \geq 3$ but not $q=n=3$. The upper bound on $A_q^\text{pr}(n,3)$ in Equation \eqref{eq:pr alpha2}, which is a consequence of the Ratio-type bound, is no worse than the upper bound from the Singleton-type bound of Theorem \ref{th: singleton pr}.
\end{proposition}

\begin{proof}
The Singleton-type bound for $d=3$ is $q^{n-2}$ if $3<1+\lceil n - \tfrac{n}{q} \rceil$, which happens exactly if $n - \tfrac{n}{q} >2 \Leftrightarrow n > 2+\tfrac{2}{q-1}$. If $q=3$, then we need $n>3$, and if $q \geq 4$, then $n \geq 3$ suffices. In these cases, we prove that the upper bound from Equation \eqref{eq:pr alpha2} is at most $q^{n-2}$. 

If $n<q$, then $\lfloor \tfrac{n}{q} \rfloor = 0$ and the upper bound from the Ratio-type bound reduces to
\[
A_q^\text{pr}(n,3) \leq q^{n-2} \frac{n(n+1)}{n(n+1)} = q^{n-2}.
\]
This exactly equals the upper bound from the Singleton-type bound for $d=3$. 

If $q \leq n < 2q$, then $\lfloor \tfrac{n}{q} \rfloor = 1$ and the upper bound from the Ratio-type bound reduces to
\[
A_q^\text{pr}(n,3) \leq q^{n-2} \frac{n(n+1)+q(-2-2n+2q)}{(n-1)n}.
\]
It can be verified with mathematical software that this is less than or equal to $q^{n-2}$ when $n \geq 2$ and $q \leq n <2q$. So the desired result holds in this case.

If $2q \leq n <3q$, then $\lfloor \tfrac{n}{q} \rfloor = 2$ and the upper bound from the Ratio-type bound reduces to
\[
A_q^\text{pr}(n,3) \leq q^{n-2} \frac{n(n+1)+2q(-2-2n+3q)}{(n-2)(n-1)}.
\]
Mathematical software can show that this is less than or equal to $q^{n-2}$ if $n \geq 3$ and $2q \leq n <3q$. Since $n \geq 2q$ and $q \geq 3$, the condition $2q \leq n < 3q$ is actually sufficient. So also in this case the desired result is reached.

Lastly, we consider the last case $n \geq 3q$. We have $q \lfloor \tfrac{n}{q} \rfloor \leq n$, so $-2-2n+q+ q \lfloor \tfrac{n}{q} \rfloor \leq -2-2n+q+n = -2-n+q < 0$ since $n \geq 3q$. Also $q \lfloor \tfrac{n}{q} \rfloor \geq q \left( \tfrac{n-(q-1)}{q} \right) = n-q+1$ since $n,q$ are integral. Then
\[
\lfloor \tfrac{n}{q} \rfloor q(-2-2n+q+ \lfloor \tfrac{n}{q} \rfloor q) \leq (n-q+1)(-2-n+q).
\]
The bound from the Ratio-type bound is thus upper bounded by
\[
A_q^\text{pr}(n,3) \leq q^n \frac{n(n+1) +(n-q+1)(-2-n+q)}{(qn-n)(qn+q-n)} = q^n \frac{2+2n-q}{n(qn+q-n)}.
\]
This is less than or equal to $q^{n-2}$ if and only if
\[
q^2(2+2n-q) \leq n(qn+q-n),
\]
which can be seen to hold, using mathematical software, for $(n,q)=(9,3)$ or $n \geq 10$ and $3 \leq q \leq \frac{n}{3}$. Since $q\geq 3$, we have $n \geq 3q \geq 9$. If $n=9$, then $3 \leq q \leq \frac{n}{3} = 3$, so $q=3$ is the only option. If $n \geq 10$, then the desired inequality holds for $n \geq 3q$. All in all, we also get the desired result when $n \geq 3q$.

Next we consider the cases where the Singleton-type bound equals $1$. This happens if $q=3$, $n=2,3$ or $q \geq 4$, $n=2$. If $n=2$, then the bound from Equation \eqref{eq:pr alpha2} reduces to 1. If $q=3$ and $n=3$, then our bound reduces to 3. So $q=3$, $n=3$ is the only case in which the upper bound on $A_q^\text{pr}(n,3)$ from Equation \eqref{eq:pr alpha2} is worse than the upper bound from the Singleton-type bound of Theorem \ref{th: singleton pr}.
\end{proof}

Also for $k=2$ and $q \geq 3$ bounds for phase-rotation codes that are almost always at least as good as the Singleton-type bound can be obtained from the Ratio-type bound. Now consider the Ratio-type bound on the 3-independence number $\alpha_3$. Again the cases $q=2$ and $q \geq 3$ are treated separately. First the case $q=2$ (and $k=3$) is studied.

\begin{theorem} \label{alpha3 q2 pr graph}
Let $n \geq 5$. Then
\[
\alpha_3(G_\text{pr}(\mathbb{F}_2^n)) \leq \begin{cases}
2^{n-1} \frac{n^2-n+4}{n^2(n+4)} & \text{if } n \equiv 0 \mod 4, \\
2^{n-1} \frac{n-3}{(n-1)(n+3)} & \text{if } n \equiv 1 \mod 4, \\
2^{n-1} \frac{n-5}{(n+2)(n-2)} & \text{if } n \equiv 2 \mod 4, \\
2^{n-1} \frac{1}{n+1} & \text{if } n \equiv 3 \mod 4.
\end{cases}
\]
\end{theorem}

\begin{proof}
Since $n \geq 5$, $G_\text{pr}(\mathbb{F}_2^n)$ has at least four distinct eigenvalues, so Theorem \ref{th: alpha3} is applicable. First we need to determine $\Delta = \max_{u \in V(G_\text{pr}(\mathbb{F}_2^n))} \{ (A^3)_{uu} \}$. Since $G_\text{pr}(\mathbb{F}_2^n)$ is walk-regular, the diagonal entries of $A^3$ are all the same, so $\Delta = (A^3)_{\textbf{00}}$. Now $\Delta$ is exactly two times the number of triangles in the graph that vertex \textbf{0} is part of. Since $n \geq 5$, vertex \textbf{0} can only be part of triangles where the vertices of the triangle differ in the same $F_i$. 
However, since $q=2$ there are no two distinct element in $\mathbb{F}_q^*$, so vertex $\textbf{0}$ is not part of any triangles, and $\Delta=0$.

We start with the case where $n$ is even, or $n \equiv 0,2 \mod 4$. Then
\[
\theta_0 = n+1, \quad \theta_r = 1-n,
\]
and $\theta_s$ is the smallest eigenvalue $\geq -\frac{\theta_0^2+\theta_0\theta_r - \Delta}{\theta_0(\theta_r+1)}$. Now
\[
\frac{\theta_0^2+\theta_0\theta_r - \Delta}{\theta_0(\theta_r+1)} = \frac{(n+1)^2+(n+1)(1-n)}{(n+1)(2-n)} = \frac{2}{2-n}.
\]
So $\theta_s:=2i-n-1$ is the smallest eigenvalue $\geq \frac{2}{n-2}$. Then
\[
2i-n-1 \geq \frac{2}{2-n} \Leftrightarrow 2i \geq n+1+\frac{2}{n-2} \Leftrightarrow i \geq \frac{n+1}{2} + \frac{1}{n-2}.
\]
Since $n \geq 5$, $\frac{1}{n-2} \leq \frac{1}{3} < \frac{1}{2}$. Since $i$ also has to be integral, we get $i \geq \frac{n+1}{2} + \frac{1}{2} = \frac{n}{2} +1$.
Since $i$ has to be odd for $2i-n-1$ to be an eigenvalue when $n$ is even, we get $i = \frac{n}{2}+1$ if $n \equiv 0 \mod 4$ and $i = \frac{n}{2}+2$ if $n \equiv 2 \mod 4$. If $n \equiv 0 \mod 4$, we have $\theta_s = 1, \theta_{s+1} = -3$. Then
\[
\alpha_3(G_\text{pr}(\mathbb{F}_2^n)) \leq 2^n \frac{-(n+1)(1-3+1-n)-(-3)(1-n)}{(n+1-1)(n+1+3)(n+1-1+n)} = 2^{n-1} \frac{n^2-n+4}{n^2(n+4)}.
\]
If $n \equiv 2 \mod 4$, we have $\theta_s = 3, \theta_{s+1} = -1$. Then
\[
\alpha_3(G_\text{pr}(\mathbb{F}_2^n)) \leq 2^n \frac{-(n+1)(3-1+1-n) - 3(-1)(1-n)}{(n+1-3)(n+1+1)(n+1-1+n)} = 2^{n-1} \frac{n-5}{(n+2)(n-2)}.
\]

Next we deal with the case where $n$ is odd, or $n \equiv 1,3 \mod 4$. Then
\[
\theta_0 = n+1, \quad \theta_r = -n-1,
\]
and $\theta_s$ is the smallest eigenvalue $\geq -\frac{\theta_0^2+\theta_0\theta_r - \Delta}{\theta_0(\theta_r+1)}$. Now
\[
\frac{\theta_0^2+\theta_0\theta_r - \Delta}{\theta_0(\theta_r+1)} = \frac{(n+1)^2+(n+1)(-n-1)}{(n+1)(-n)} = 0.
\]
So $\theta_s:=2i-n-1$ is the smallest eigenvalue $\geq 0$. Then
\[
2i-n-1 \geq 0 \Leftrightarrow 2i \geq n+1 \Leftrightarrow i \geq \frac{n+1}{2}.
\]
Since $i$ has to be even for $2i-n-1$ to be an eigenvalue when $n$ is odd, we get $i = \frac{n}{2}+1$ if $n \equiv 1 \mod 4$ and $i = \frac{n+1}{2}$ if $n \equiv 3 \mod 4$. If $n \equiv 1 \mod 4$, we have $\theta_s = 2, \theta_{s+1} = -2$. Then
\[
\alpha_3(G_\text{pr}(\mathbb{F}_2^n)) \leq 2^n \frac{-(n+1)(2-2-n-1)-2(-2)(-n-1)}{(n+1-2)(n+1+2)(n+1+n+1)} = 2^{n-1} \frac{n-3}{(n-1)(n+3)}.
\]
If $n \equiv 3 \mod 4$, we have $\theta_s = 0, \theta_{s+1} = -4$. Then
\begin{align*}
&\alpha_3(G_\text{pr}(\mathbb{F}_2^n)) \leq 2^n \frac{-(n+1)(0-4-n-1)-0(-4)(-n-1)}{(n+1)(n+1+4)(n+1+n+1)} = \frac{2^{n-1}}{n+1}.
\qedhere\end{align*}
\end{proof}

The upper bounds from Theorem \ref{alpha3 q2 pr graph} can be translated to upper bounds on $A^\text{pr}_2(n,4)$ via Lemma \ref{lemma: geo dist then max card alphak}. 
\begin{corollary} \label{cor: pr alpha3 q2}
The maximum cardinality of phase-rotation codes in $\mathbb{F}_2^n$ of minimum distance 4 with $n \geq 5$ is upper bounded by

\begin{equation} \label{eq:pr alpha3 q2}
A_2^\text{pr}(n,4) \leq \begin{cases}
2^{n-1} \frac{n^2-n+4}{n^2(n+4)} & \text{if } n \equiv 0 \mod 4, \\
2^{n-1} \frac{n-3}{(n-1)(n+3)} & \text{if } n \equiv 1 \mod 4, \\
2^{n-1} \frac{n-5}{(n+2)(n-2)} & \text{if } n \equiv 2 \mod 4, \\
2^{n-1} \frac{1}{n+1} & \text{if } n \equiv 3 \mod 4.
\end{cases}
\end{equation}
\end{corollary}
A comparison of these upper bounds from Corollary \ref{cor: pr alpha3 q2} to the Singleton-type bound from Theorem \ref{th: singleton pr} follows next.

\begin{proposition}
Let $n \geq 5$. The upper bounds on $A_2^\text{pr}(n,4)$ in Equation \eqref{eq:pr alpha3 q2}, which are a result of the Ratio-type bound, are no worse than the upper bound from the Singleton-type bound of Theorem \ref{th: singleton pr}.
\end{proposition}

\begin{proof}
The upper bound of the Singleton-type bound for $d=4$ and $q=2$ is $2^{n-3}$ if $4<1+\lceil n - \tfrac{n}{2} \rceil$, which is exactly if $\tfrac{n}{2} > 3 \Leftrightarrow n >6$. In this case we compare the bounds and see when the bounds from Equation \eqref{eq:pr alpha3 q2} are smaller than or equal to $2^{n-3}$. For $n \equiv 0 \mod 4$ we have:
\[
2^{n-1} \frac{n^2-n+4}{n^2(n+4)} \leq 2^{n-3} \Leftrightarrow 2^2(n^2-n+4) \leq n^2(n+4) \Leftrightarrow n^3+4n \geq 16,
\]
which is true since $n \geq 5$. For $n \equiv 1 \mod 4$:
\[
2^{n-1} \frac{n-3}{(n-1)(n+3)} \leq 2^{n-3} \Leftrightarrow 2^2(n-3) \leq (n-1)(n+3) \Leftrightarrow n^2-2n+9 = (n-1)^2+8 \geq 0.
\]
Also this is true. For $n \equiv 2 \mod 4$:
\[
2^{n-1} \frac{n-5}{(n+2)(n-2)} \leq 2^{n-3} \Leftrightarrow 2^2(n-5) \leq (n-2)(n+2) \Leftrightarrow n^2-4n+16 = (n-2)^2+12 \geq 0,
\]
which is true. Lastly, for $n \equiv 3 \mod 4$ we get:
\[
2^{n-1} \frac{1}{n+1} \leq 2^{n-3} \Leftrightarrow 2^2 \leq n+1 \Leftrightarrow n \geq 3,
\]
which is true by assumption on $n$. 

In the cases that the Singleton-type bound equals $1$, which is if $n=5,6$, the bounds from Equation \eqref{eq:pr alpha3 q2} give values of $2^4 \cdot \tfrac{2}{32} = 1$ and $2^5 \cdot \tfrac{1}{32} = 1$ for $n=5,6$ respectively. Hence the upper bounds on $A_2^\text{pr}(n,4)$ from Equation \eqref{eq:pr alpha3 q2} are no worse than the Singleton-type upper bound from Theorem \ref{th: singleton pr}.
\end{proof}

So the Ratio-type bound gives upper bounds on the size of phase-rotation codes that perform no worse than the Singleton-type bound for $k=3$ and $q=2$. Now we consider the Ratio-type bound for $k=3$ and $q \geq 3$. 

\begin{theorem}\label{th: alpha3 q3 pr graph}
Let $n \geq 3$ and $q \geq 3$. Then
\[
\alpha_3(G_\text{pr}(\Fq^n)) \leq q^n \frac{n(n+2q-1)+q \lceil \tfrac{n-1}{q} \rceil \big(-2n-q+q \lceil \tfrac{n-1}{q} \rceil \big)}{q^3 \big( n+\lfloor \tfrac{1-n}{q} \rfloor \big) \big( n+1+ \lfloor \tfrac{1-n}{q} \rfloor \big)}.
\]
\end{theorem}

\begin{proof}
The eigenvalues of $G_\text{pr}(\Fq^n)$ for $n\geq 3$ and $q \geq 3$ are $iq-n-1$ for $i=0,1, \ldots, n-1,n+1$. Since $n\geq 3$, there are at least four distinct eigenvalues. Since $G_\text{pr}(\Fq^n)$ is regular, Theorem \ref{th: alpha3} is applicable. First we need to determine $\Delta = \max_{u \in V(G_\text{pr}(\Fq^n))} \{ (A^3)_{uu} \}$. Similarly to the previous proof, $\Delta = (A^3)_{\textbf{00}}$, which equals two times the number of triangles that vertex $\textbf{0}$ is part of. 
Again since $n \geq 3$, vertex \textbf{0} is only part of triangles where the vertices of the triangle differ in the same $F_i$. 
Then \textbf{0} is part of $(n+1) {q-1 \choose 2}$ triangles since there are $n+1$ $F_i$'s and ${q-1 \choose 2}$ ways to choose two different elements in $\Fq^*$. So $\Delta = 2(n+1){q-1 \choose 2} = (n+1)(q-1)(q-2)$. Using the notation of Theorem \ref{th: alpha3} we have
\[
\theta_0=(n+1)(q-1), \quad \theta_r = -n-1,
\]
and $\theta_s$ is the smallest eigenvalue $\geq - \frac{\theta_0^2+\theta_0\theta_r - \Delta}{\theta_0(\theta_r+1)}$. Now 
\[
\frac{\theta_0^2+\theta_0\theta_r - \Delta}{\theta_0(\theta_r+1)} = \frac{(n+1)^2(q-1)^2 + (n+1)(q-1)(-n-1) - (n+1)(q-1)(q-2)}{(n+1)(q-1) \cdot -n} = 2-q.
\]
So $\theta_s:=iq-n-1$ is the smallest eigenvalue $\geq q-2$. Then
\[
iq-n-1 \geq q-2 \Leftrightarrow (i-1)q \geq n-1 \Leftrightarrow i \geq 1+ \frac{n-1}{q}.
\]
Since $i$ has to be equal to one of the integers $0,1, \ldots, n-1,n+1$, take $i = 1+ \lceil \frac{n-1}{q} \rceil$, which is a positive integer and at most $n-1$. Then 
\[
\theta_s = (1+ \lceil \tfrac{n-1}{q} \rceil)q-n-1, \theta_{s+1} = \lceil \tfrac{n-1}{q} \rceil q-n-1.
\]
Now we obtain the following upper bound for $\alpha_3(G_\text{pr}(\Fq^n))$:
{\scriptsize
\[
q^n \Biggl( \frac{(n+1)(q-1)(q-2) - (n+1)(q-1) ((1+ \lceil \tfrac{n-1}{q} \rceil)q-n-1 +  \lceil \tfrac{n-1}{q} \rceil q-n-1 -n-1) }{((n+1)(q-1) - ((1+ \lceil \tfrac{n-1}{q} \rceil)q-n-1))((n+1)(q-1) - ( \lceil \tfrac{n-1}{q} \rceil q-n-1))((n+1)(q-1) +n+1)} 
\]
\[
- \frac{((1+ \lceil \tfrac{n-1}{q} \rceil)q-n-1) ( \lceil \tfrac{n-1}{q} \rceil q-n-1)(-n-1)}{((n+1)(q-1) - ((1+ \lceil \tfrac{n-1}{q} \rceil)q-n-1))((n+1)(q-1) - ( \lceil \tfrac{n-1}{q} \rceil q-n-1))((n+1)(q-1) +n+1)}  \Biggl)
\]
\[
q^n \frac{n(n+2q-1)+q \lceil \tfrac{n-1}{q} \rceil \big(-2n-q+q \lceil \tfrac{n-1}{q} \rceil \big)}{q^3 \big( n+\lfloor \tfrac{1-n}{q} \rfloor \big) \big( n+1+ \lfloor \tfrac{1-n}{q} \rfloor \big)}
\]}
\end{proof}

The upper bound from Theorem \ref{th: alpha3 q3 pr graph} can be translated to upper bounds on $A_q^\text{pr}(n,4)$ via Lemma \ref{lemma: geo dist then max card alphak}. 
\begin{corollary} \label{cor: alpha 3 pr}
The maximum cardinality of phase-rotation codes in $\Fq^n$ of minimum distance 4 with $n \geq 3$ and $q \geq 3$ is upper bounded by
\begin{equation} \label{eq: alpha 3 pr}
A_q^\text{pr}(n,4) \leq q^n \frac{n(n+2q-1)+q \lceil \tfrac{n-1}{q} \rceil \big(-2n-q+q \lceil \tfrac{n-1}{q} \rceil \big)}{q^3 \big( n+\lfloor \tfrac{1-n}{q} \rfloor \big) \big( n+1+ \lfloor \tfrac{1-n}{q} \rfloor \big)}.
\end{equation}
\end{corollary}
We compare this upper bound from Corollary \ref{cor: alpha 3 pr} with the Singleton-type upper bound from Theorem \ref{th: singleton pr}.

\begin{proposition} \label{prop: Aq(n,4) pr}
Let $n \geq 3$, $q\geq 3$ but not $(n,q)= (4,3)$ or $(n,q)= (4,4)$. The upper bound on $A_q^\text{pr}(n,4)$ in Equation \eqref{eq: alpha 3 pr}, which is obtained from the Ratio-type bound, is no worse than the upper bound from the Singleton-type bound of Theorem \ref{th: singleton pr}.
\end{proposition}

\begin{proof}
The upper bound from the Singleton-type bound for $d=4$ is $q^{n-4+1}=q^{n-3}$ if $4<1+\lceil n - \tfrac{n}{q} \rceil$. This is exactly when $n-\tfrac{n}{q} > 3 \Leftrightarrow n>3+\tfrac{3}{q-1}$. If $q=3,4$, then we need $n \geq 5$, and if $q \geq5$, then $n \geq 4$ suffices. We consider these cases first. Define $m:= \lceil \tfrac{n-1}{q} \rceil$. Then $m-1 < \tfrac{n-1}{q} \leq m$ by definition of $m$. Since $\lfloor \tfrac{1-n}{q} \rfloor = -\lceil \tfrac{n-1}{q} \rceil = -m$, the upper bound from Equation \eqref{eq: alpha 3 pr} becomes:
\[
q^n \frac{n(n+2q-1)+q m (-2n-q+q m)}{q^3(n-m)(n+1-m)}.
\]
The latter is less than or equal to the Singleton-type upper bound if
\[
\frac{n(n+2q-1)+q m (-2n-q+q m)}{(n-m)(n+1-m)} \leq 1
\]
\[
\Leftrightarrow n(n+2q-1)+qm(-2n-q+qm) \leq (n-m)(n+1-m)
\]
\[
\Leftrightarrow 2qn-n-2q m n -q^2m+q^2 m^2 \leq n -2mn -m +m^2.
\]
Since $m,n,q$ are integral, the latter inequality can be shown to hold, using mathematical software, when $n \geq 3, q \geq 3$, and $\smash{m-1 < \tfrac{n-1}{q} \leq m}$.
Note that by definition of $m$, we have $\smash{m-1 < \tfrac{n-1}{q} \leq m}$. The conditions on $n$ and $q$ hold by the given assumptions on $n$ and $q$.

Next we consider the cases $q=3,4$, $n=3,4$ and $q \geq 5$, $n=3$. Now the Singleton-type bound gives an upper bound of $1$. If $n=3$, the upper bound from Equation \eqref{eq: alpha 3 pr} reduces to $1$. However for $n=4$, $q=3,4$, the upper bound of Equation \eqref{eq: alpha 3 pr} reduces to $q$, which is not less than or equal to $1$. That finishes the proof.
\end{proof}


So upper bounds on the size of phase-rotation codes obtained via the Ratio-type bound almost always perform no worse than the Singleton-type bound for $k=3$ and $q \geq 3$. 

We have shown theoretically that, for the phase-rotation metric, the Ratio-type bound performs no worse than the Singleton-type bound in most cases when the minimum distance is small, i.e. $d=2,3,4$, and $n$ is large enough. Next we provide some computational results for larger values of the minimum distance. 
Consider all graphs $G_\text{pr}(\Fq^n)$ with $n \geq 2$, $q$ a prime power and at most 1000 vertices, and consider $k=1,\ldots, \lceil \tfrac{q-1}{q}n \rceil-1$. We compare the Inertia-type bound, the Ratio-type bound, and the Singleton-type bound in these instances. Note that the graph $G_\text{pr}(\Fq^n)$ is not explicitly constructed, but its eigenvalues are calculated using Proposition \ref{prop: eigval pr graph}. This implies that a comparison of the upper bounds to the actual value of the $k$-independence number is not possible. For all the considered instances where moreover $d<1+\lceil n -\tfrac{n}{q} \rceil$, the Ratio-type bound performs no worse than the Singleton-type bound. In some instances, like $n=5,6,q=3,k=3$ and $n=9,q=2,k=3,4$, the Ratio-type bound improves on the Singleton-type upper bound. There are also some improvements with the Inertia-type bound compared to the Singleton-type. If $n=6,q=2,k=1$ or $n=8, q=2, k=1,3$, the Inertia-type bound performs better than the Singleton-type bound and the Ratio-type bound.

In what follows, we show some more results for the Inertia-type bound and the Ratio-type bound. This time the graph $G_\text{pr}(\Fq^n)$ is explicitly constructed, so the upper bounds on the $k$-independence number can be compared to the true $k$-independence number. The results for 
\[
n=2,3,4, \quad q=2,3,4,5, \quad k=1,\ldots,  \lceil \tfrac{q-1}{q}n \rceil-1,
\]
\[
\text{and } n=5, \quad q=2,3, \quad k=1,\ldots,  \lceil \tfrac{q-1}{q}n \rceil-1
\]
can be seen in Table \ref{tab: phase rotation results}. The columns ``Inertia-type'' and ``Ratio-type'' contain the value of the Inertia-type bound and the value of the Ratio-type bound, respectively, for the given graph instance. Similarly the column ``$\vartheta(G^k)$'' contains the value of the Lovász theta number and the column ``Singleton-type'' contains the value of the Singleton-type upper bound. Since it is computationally expensive to compute the Lovász theta number, the value is not computed for every graph instance. In that case, this is indicated by a dash in the corresponding entry of the table. The column ``$\alpha_k$'' contains the value of the true $k$-independence number of that graph instance. Only the instances where the Inertia-type bound or the Ratio-type bound performed no worse than the Singleton-type bound are provided. A value in the columns ``Inertia-type'' and ``Ratio-type'' is indicated in bold when it is lower than the corresponding Singleton-type upper bound.

\begin{table}[ht]
\centering
\begin{tabular}{|ccc|c|c|c|c|c|}
\hline
$q$ & $n$ & $k$ & Inertia-type & Ratio-type & $\alpha_k$ & $\vartheta(G^k)$ & \makecell{Singleton-type \\ (Theorem \ref{th: singleton pr})} \\
\hline
\hline
2&2&1&1&1&1&1&1 \\
3&2&1&7&3&3&3&3 \\
4&2&1&6&4&4&4&4 \\
5&2&1&12&5&5&5&5 \\
2&3&1&7&4&4&4&4 \\
2&3&2&1&1&1&1&1 \\
3&3&1&13&9&9&9&9 \\
4&3&1&43&16&16&16&16 \\
4&3&2&19&4&4&4&4 \\
5&3&1&52&25&25&25&25 \\
5&3&2&25&5&5&5&5 \\
2&4&1&\textbf{5}&\textbf{6}&5&6&8 \\
2&4&2&1&1&1&-&1 \\
2&4&3&1&1&1&-&1 \\
3&4&1&40&27&27&27&27 \\
3&4&2&11&\textbf{6}&6&6&9 \\
4&4&1&91&64&64&64&64 \\
4&4&2&61&16&16&-&16 \\
5&4&1&421&125&125&-&125 \\
5&4&2&161&25&25&-&25 \\
5&4&3&41&5&5&-&5 \\
2&5&1&16&16&16&-&16 \\
2&5&2&\textbf{2}&\textbf{2}&2&-&8 \\
2&5&3&1&1&1&-&1 \\
2&5&4&1&1&1&-&1 \\
3&5&1&161&81&81&-&81 \\
3&5&2&53&\textbf{16}&11&-&27 \\
3&5&3&22&\textbf{6}&6&-&9 \\
\hline
\end{tabular}
\caption{Results of the Inertia-type bound and the Ratio-type bound for the phase-rotation metric, compared to the Singleton-type bound, the Lovász theta number $\vartheta(G^k)$, and the actual $k$-independence number $\alpha_k$. Improvements of the Inertia-type bound and the Ratio-type bound compared to the Singleton-type bound are in bold.}
\label{tab: phase rotation results}
\end{table}

As expected from the theoretical results, the Ratio-type bound performs well. It performs at least as good as the Singleton-type bound in almost all tested instances, even in an instance with $k = 4$, which was not included in the earlier theoretical analysis. Moreover, the Ratio-type bound improves on the Singleton-type bound in several instances and is also sharp in many instances. The performance of the Inertia-type bound, on the other hand, varies widely. In most instances it performs worse than the Ratio-type bound. However, when $n=5, q=2, k=1,2$ the Inertia-type bound performs equally good and is sharp. Moreover, when $n=4, q=2, k=1$ the Inertia-type bound outperforms both the Ratio-type bound and the Singleton-type bound, and is equal to the $k$-independence number.

To sum up the results for the phase-rotation metric, we have seen that the spectral bounds improve the Singleton-type bound in several instances. For the Ratio-type bound it was proven theoretically that in most instances where the minimum distance is small (and $n$ is large enough) the Ratio-type bound is at least as good as the Singleton-type bound. Some computational results also show improvement for the Ratio-type bound in several instances. For the Inertia-type bound, computational results show that there are a few instances where it outperforms the Ratio-type bound and the Singleton-type bound, while its overall performance varies widely.

\section{Tightness results for other metrics} \label{sec: equality metrics}

In this section we apply the Eigenvalue Method to three more metrics; the block metric, the cyclic $b$-burst metric, and the Varshamov metric. While the bounds obtained from this method do not improve state-of-the-art bounds for any of these metrics, there are specific instances where the Eigenvalue Method gives tight bounds; see Table \ref{tab:metrics overview}. In these instances the Inertia-type bound or the Ratio-type bound equals the $k$-independence number, and thus equals the maximum cardinality of codes of a specific minimum distance. So the Eigenvalue Method gives an alternative approach for calculating the maximum cardinality of codes in the block metric, the cyclic $b$-burst metric, and the Varshamov metric.

\subsection{Block metric} \label{sec: block}
The block metric was introduced in \cite{Feng2006LinearCodes}.

\begin{definition}
Let $P=\{p_1, \ldots, p_m\}$ be a partition of $[n]$. The \textit{block $P$-weight} of $\textbf{x} \in \Fq^n$ is defined as
\[
w_P(\textbf{x}) := \min\left\{|I|: \text{supp}(\textbf{x}) \subseteq \bigcup_{i \in I} p_i\right\}.
\]
The \textit{block $P$-distance} between $\textbf{x},\textbf{y} \in \Fq^n$ is defined as $d_P(\textbf{x},\textbf{y}) := w_P(\textbf{x}-\textbf{y})$.
\end{definition}

Fix a partition $P = \{p_1, \dots, p_m\}$ of $[n]$. Applying the Eigenvalue Method to the discrete metric space $(\Fq^n, d_P)$ gives the \textit{block $P$-distance graph} $G_P(\Fq^n)$. This graph satisfies condition \ref{item: C1} and properties \ref{item: C2} and \ref{item: C3}. Moreover, property \ref{item: C4} holds since $G_P(\Fq^n)$ is not distance-regular in general. So both the Inertia-type bound and the Ratio-type bound, and their respective linear programs, can be applied to this graph.


The bounds obtained via the Eigenvalue Method can be compared to a Singleton-type bound: for a code $\mathcal{C} \subseteq \Fq^n$ of minimum block $P$-distance $d$ it holds that
\begin{equation} \label{eq: singleton type block metric}
|\mathcal{C}| \leq q^{\sum_{j=d}^m \, p_j},
\end{equation}
where w.l.o.g. $|p_1| \geq \cdots \geq |p_m|$. This bound can easily be derived from the Singleton-type bound for the combinatorial metric, which can be found in \cite{Bossert1996Singleton-TypeCodes}, since the block metric is an example of a combinatorial metric. 
Now we test the performance of the following instances:
\[
P=\big\{ \{1,2\},\{3\} \big\}, \big\{ \{1,2\},\{3,4\} \big\}, \big\{ \{1,2,3\},\{4,5\} \big\}, \quad q=2,3,4, \quad k=1,\ldots, m-1,
\]
\[
\text{and } P=\big\{ \{1,2\},\{3,4\},\{5,6\} \big\}, \big\{ \{1,2,3\},\{4,5\},\{6\} \big\}, \quad q=2,3, \quad k=1,\ldots, m-1.
\]
The results can be seen in Table \ref{tab: block results}. The columns ``Inertia-type'', ``Ratio-type'', and ``Singleton-type'' give the value of the Inertia-type bound, the Ratio-type bound and the Singleton-type bound, respectively, for the given instance. The columns ``$\alpha_k$'' and ``$\vartheta(G^k)$'' contain the value of the $k$-independence number and the value of the Lovász theta number, respectively. For some instances the Lovász theta number could not be calculated in reasonable time, which is indicated by a dash in the table. Since the Singleton-type bound always performs at least as good as the bounds obtained using the Eigenvalue Method, only the instances where either the Inertia-type bound or the Ratio-type bound attains the $k$-independence number are displayed in the table. 

\begin{table}[ht] 
\centering
\begin{tabular}{|ccc|c|c|c|c|c|}
\hline
$P$ & $q$ & $k$ & Inertia-type & Ratio-type & $\alpha_k$ & $\vartheta(G^k)$ & \makecell{Singleton-type \\ (Equation \eqref{eq: singleton type block metric}) } \\
\hline
$\big\{ \{1,2\},\{3\} \big\}$ & 2 & 1 & 5 & 2 & 2 & 2 & 2 \\
$\big\{ \{1,2\},\{3,4\} \big\}$ & 2 & 1 & 7 & 4 & 4 & 4 & 4 \\
$\big\{ \{1,2\},\{3,4\} \big\}$ & 3 & 1 & 17 & 9 & 9 & 9 & 9 \\
$\big\{ \{1,2\},\{3,4\} \big\}$ & 4 & 1 & 31 & 16 & 16 & 16 & 16 \\
$\big\{ \{1,2\},\{3,4\},\{5,6\} \big\}$ & 2 & 1 & 27 & 16 & 16 & 16 & 16 \\
$\big\{ \{1,2\},\{3,4\},\{5,6\} \big\}$ & 2 & 2 & 10 & 4 & 4 & 4 & 4 \\
$\big\{ \{1,2\},\{3,4\},\{5,6\} \big\}$ & 3 & 1 & 217 & 81 & 81 & - & 81 \\
$\big\{ \{1,2\},\{3,4\},\{5,6\} \big\}$ & 3 & 2 & 25 & 9 & 9 & - & 9 \\

\hline
\end{tabular}
\caption{Results of the Inertia-type bound and the Ratio-type for the block metric, compared to the Singleton-type bound, the Lovász theta number $\vartheta(G^k)$, and the actual $k$-independence number $\alpha_k$.}
\label{tab: block results}
\end{table}

We see that the Ratio-type bound equals the $k$-independence number is some specific instances, while the Inertia-type bound is strictly larger in all tested instances. Notably, most instances where the Ratio-type bound is tight are of the form $|p_1|= \cdots = |p_m|$.

\subsection{Cyclic $b$-burst metric} \label{sec: cyclic b burst}
The cyclic $b$-burst metric was introduced in \cite{Bridwell1970BurstCorrection}.

\begin{definition} \label{def: cyclic b burst}
Let $2 \leq b \leq n-1$. Define $A_i := \{i+j: j=1, \ldots, b\}$ for $i=0, \ldots, n-1$, where the addition is done modulo $n$ and $0 \mod{n}$ is denoted as $n$. Let $\mathcal{A} := \{A_0, \ldots, A_{n-1}\}$. The cyclic $b$-burst weight of $\textbf{x} \in \Fq^n$ is defined as
\[
w_b(\textbf{x}) := \min\left\{|I|: \text{supp}(\textbf{x}) \subseteq \bigcup_{i \in I} A_i\right\}.
\]
The cyclic $b$-burst distance between $\textbf{x}, \textbf{y} \in \Fq^n$ is defined as $d_b(\textbf{x}, \textbf{y}) := w_b(\textbf{x} - \textbf{y})$.
\end{definition}

\begin{example}
To illustrate the set $\mathcal{A}$ from Definition \ref{def: cyclic b burst}, consider $n=5$ and $b=3$. Then
\[
A_0=\{1,2,3\}, \quad A_1=\{2,3,4\}, \quad A_2=\{3,4,5\}, \quad A_3=\{4,5,1\}, \quad A_4=\{5,1,2\}.
\]
\end{example}


Fix $2 \leq b \leq n-1$. Applying the Eigenvalue Method to the discrete metric space $(\Fq^n,d_b)$ gives the \textit{cyclic $b$-burst distance graph} $G_b(\Fq^n)$. This graph satisfies condition \ref{item: C1} and properties \ref{item: C2} and \ref{item: C3}. Moreover, $G_b(\Fq^n)$ is not distance-regular in general, so property \ref{item: C4} holds. This means both the Inertia-type bound and the Ratio-type bound, and their respective linear programs, can be applied to this graph.

The bounds obtained via the Eigenvalue Method can be compared to a Singleton-type bound: for a code $\mathcal{C} \subseteq \Fq^n$ of minimum cyclic $b$-burst distance $d$ it holds that
\begin{equation} \label{eq: singleton type cyclic b burst}
|\mathcal{C}| \leq q^{n-b(d-1)}.
\end{equation}
This bound can be derived from the Singleton-type bound for the combinatorial metric in \cite{Bossert1996Singleton-TypeCodes}, since the cyclic $b$-burst metric is an example of a combinatorial metric. This bound is also known as the extended Reiger bound (see \cite{Villalba2016OnCodes}).
We test the performance of the spectral bounds in the following instances:
\[
n=3,4,5, \quad q=2,3, \quad b=2,\ldots, n-1, \quad k=1, \ldots \lceil \tfrac{n}{b} \rceil -1,
\]
\[
\text{and } n=3,4, \quad q=5, \quad b=2, \quad k=1.
\]
The results can be seen in Table \ref{tab: cyclic b burst results}. The columns ``Inertia-type'', ``Ratio-type'', and ``Singleton-type'' give the value of the Inertia-type bound, the Ratio-type bound and the Singleton-type bound, respectively, for the given instance. The columns ``$\alpha_k$'' and ``$\vartheta(G^k)$'' contain the value of the $k$-independence number and the value of the Lovász theta number, respectively. Since the Singleton-type bound always performs at least as good as the bounds obtained using the Eigenvalue Method, only the instances where either the Inertia-type bound or the Ratio-type bound attains the $k$-independence number are displayed in the table.

\begin{table}[ht] 
\centering
\begin{tabular}{|cccc|c|c|c|c|c|}
\hline
$n$ & $q$ & $b$ & $k$ & Inertia-type & Ratio-type & $\alpha_k$ & $\vartheta(G^k)$ & \makecell{Singleton-type \\ (Equation \eqref{eq: singleton type cyclic b burst})} \\
\hline
3&2&2&1&5&2&2&2&2 \\
3&3&2&1&15&3&3&3&3 \\
4&2&3&1&9&2&2&2&2 \\
5&2&2&2&6&2&2&2&2 \\
5&2&4&1&17&2&2&2&2 \\

\hline
\end{tabular}
\caption{Results of the Inertia-type bound and the Ratio-type for the cyclic $b$-burst metric, compared to the Singleton-type bound, the Lovász theta number $\vartheta(G^k)$, and the actual $k$-independence number $\alpha_k$.}
\label{tab: cyclic b burst results}
\end{table}

Table \ref{tab: cyclic b burst results} shows that the Inertia-type bound is not tight in any tested instances, while the Ratio-type bound is tight in some instances. However, it is not immediately clear why those instances give a tightness for the Ratio-type bound.

\subsection{Varshamov metric} \label{sec: varshamov}
The Varshamov metric, also known as the asymmetric metric, was introduced in \cite{Varshamov1965OnCodes}.

\begin{definition}
The \textit{Varshamov distance} between $\textbf{x},\textbf{y} \in \mathbb{F}_2^n$ is defined as 
\[
d_\text{Var}(\textbf{x},\textbf{y}) := \tfrac{1}{2} \left( w_\text{H}(\textbf{x}-\textbf{y}) + \left|w_\text{H}(\textbf{x})- w_\text{H}(\textbf{y})\right| \right),
\]
where $w_\text{H}$ denotes the Hamming weight. 
\end{definition}

Another definition of the Varshamov distance between $\textbf{x}=(x_1,\ldots,x_n),\textbf{y}=(y_1,\ldots,y_n) \in \mathbb{F}_2^n$ is given in \cite{Rao1975AsymmetricMemories}: 
\[
d_\text{Var}(\textbf{x},\textbf{y}) := \max \{N_{01}(\textbf{x},\textbf{y}), N_{10}(\textbf{x},\textbf{y})\},
\]
where 
\[
N_{01}(\textbf{x},\textbf{y}) := \left|\{i: x_i = 0, y_i = 1\} \right|, \qquad N_{10}(\textbf{x},\textbf{y}) := \left|\{i: x_i = 1, y_i = 0\}\right|.
\]
In  \cite[Lemma 2.1]{Klve1981ErrorChannel} the equivalence of both definitions is proven.

Applying the Eigenvalue Method to the discrete metric space $(\mathbb{F}_2^n, d_\text{Var})$ gives the \textit{Varshamov distance graph} $G_\text{Var}(\mathbb{F}_2^n)$. This graph satisfies condition \ref{item: C1} and property \ref{item: C4}. However, desired properties \ref{item: C2} and \ref{item: C3} do not hold. This means only the Inertia-type bound can be applied to this graph.

The bound obtained via the Eigenvalue Method can be compared to a Plotkin-type bound \cite{Borden1983ACodes} and to a bound due to Varshamov \cite{Varshamov1964EstimateRussian}. The latter bound states that for a code $\mathcal{C} \subseteq \mathbb{F}_2^n$ of minimum Varshamov distance $d$ it holds that
\begin{equation} \label{eq: varshamov bound}
|\mathcal{C}| \leq \frac{2^{n+1}}{\sum\limits_{i=0}^{d-1} {\lfloor n/2 \rfloor \choose i} + {\lceil n/2\rceil \choose i}}.
\end{equation}
Note that an integer programming bound also exists for codes in the Varshamov metric (see e.g. \cite{Delsarte1981BoundsCodes}).
We test some instances of the graph, specifically
\[
n=2,\ldots, 8, \quad k=1,\ldots n-1 \text{ except } (n,k)=(8,7).
\]
The results can be seen in Table \ref{tab: Varshamov results}. The columns ``Inertia-type'', ``Plotkin-type'', and ``Varshamov'' give the value of the Inertia-type bound, the Plotkin-type bound and the bound due to Varshamov, respectively, for the given instance. The columns ``$\alpha_k$'' and ``$\vartheta(G^k)$'' contain the value of the $k$-independence number and the value of the Lovász theta number, respectively. Since either the Plotkin-type bound or the bound due to Varshamov always performs at least as good as the Inertia-type bound, only the instances where the Inertia-type bound attains the $k$-independence number are displayed in the table. 

\begin{table}[ht] 
\centering
\begin{tabular}{|cc|c|c|c|c|c|}
\hline
$n$ & $k$ & Inertia-type & $\alpha_k$ & $\vartheta(G^k)$ & \makecell{Plotkin-type \cite{Borden1983ACodes}} & \makecell{Varshamov \\ (Equation \eqref{eq: varshamov bound})} \\
\hline
2&1&2&2&2.0&2&2 \\
3&2&2&2&2.0&2&2 \\
4&3&2&2&2.0&2&4 \\
5&3&2&2&2.0&2&5 \\
5&4&2&2&2.0&2&5 \\
6&4&2&2&2.0&2&8 \\
6&5&2&2&2.0&2&8 \\
7&5&2&2&2.0&2&10 \\
7&6&2&2&2.0&2&10 \\

\hline
\end{tabular}
\caption{Results of the Inertia-type bound for the Varshamov metric, compared to the Plotkin-type bound, the bound from Varshamov, the Lovász theta number $\vartheta(G^k)$, and the actual $k$-independence number $\alpha_k$.}
\label{tab: Varshamov results}
\end{table}

We can see in Table \ref{tab: Varshamov results} that the instances where the Inertia-type bound is tight are those where $k$ is close to $n$. In all these instances the $k$-independence number equals 2.

\subsection*{Acknowledgements}
Aida Abiad is supported by the Dutch Research Council through the grants VI.Vidi.213.085 and OCENW.KLEIN.475. Alberto Ravagnani is supported by the Dutch Research Council through the grants VI.Vidi.203.045, OCENW.KLEIN.539, and by the Royal Academy of Arts and Sciences of the Netherlands.

\bibliographystyle{abbrv}
\bibliography{references2.bib}

\begin{thebibliography}{10}

\bibitem{Abiad2024EigenvalueCodes}
A.~Abiad, G.~N. Alfarano, and A.~Ravagnani.
\newblock {Eigenvalue bounds and alternating rank-metric codes}.
\newblock {\em Journal of Algebra and its Applications}, to appear, 2025.

\bibitem{Abiad2019OnGraphs}
A.~Abiad, G.~Coutinho, and M.~A. Fiol.
\newblock {On the k-independence number of graphs}.
\newblock {\em Discrete Mathematics}, 342(10):2875--2885, 2019.

\bibitem{Abiad2022OptimizationPowers}
A.~Abiad, G.~Coutinho, M.~A. Fiol, B.~D. Nogueira, and S.~Zeijlemaker.
\newblock {Optimization of eigenvalue bounds for the independence and chromatic
  number of graph powers}.
\newblock {\em Discrete Mathematics}, 345(3):112706, 2022.

\bibitem{Abiad2024ACodes}
A.~Abiad, A.~L. Gavrilyuk, A.~P. Khramova, and I.~Ponomarenko.
\newblock {A linear programming bound for sum-rank metric codes}.
\newblock {\em IEEE Transactions on Information Theory}, 71:317--329, 2025.

\bibitem{Abiad2023EigenvalueCodes}
A.~Abiad, A.~P. Khramova, and A.~Ravagnani.
\newblock {Eigenvalue Bounds for Sum-Rank-Metric Codes}.
\newblock {\em IEEE Transactions on Information Theory}, 70:4843--48554, 2024.

\bibitem{Abiad2024EigenvalueCodesb}
A.~Abiad, A.~Neri, and L.~Reijnders.
\newblock {Eigenvalue bounds for the distance-$t$ chromatic number of a graph
  and their application to Lee codes}.
\newblock {\em Journal of Algebra and its Applications}, to appear, 2025.

\bibitem{Astola1982TheLee-codes}
J.~Astola.
\newblock {The Lee-scheme and bounds for Lee-codes}.
\newblock {\em Cybernetics and Systems}, 13(4):331--343, 10 1982.

\bibitem{Ball2010}
S.~Ball.
\newblock On sets of vectors of a finite vector space in which every subset of
  basis size is a basis.
\newblock {\em Journal of the European Mathematical Society (JEMS)},
  14(3):733--748, 2010.

\bibitem{BallDeBeule2012a}
S.~Ball and J.~De~Beule.
\newblock On sets of vectors of a finite vector space in which every subset of
  basis size is a basis ii.
\newblock {\em Designs, Codes and Cryptography}, 65(1-2):5--14, 2012.

\bibitem{Basavaraju2014RainbowProducts}
M.~Basavaraju, L.~S. Chandran, D.~Rajendraprasad, and A.~Ramaswamy.
\newblock {Rainbow Connection Number of Graph Power and Graph Products}.
\newblock {\em Graphs and Combinatorics}, 30(6):1363--1382, 2014.

\bibitem{Berlekamp1968TheCodes}
E.~R. Berlekamp.
\newblock {\em Algebraic Coding Theory}.
\newblock World Scientific Pub Co Inc, revised edition, 2015.

\bibitem{Borden1983ACodes}
J.~Borden.
\newblock {A low-rate bound for asymmetric error-correcting codes}.
\newblock {\em IEEE Transactions on Information Theory}, 29(4):600--602, 1983.

\bibitem{Bossert1996Singleton-TypeCodes}
M.~Bossert and V.~Sidorenko.
\newblock {Singleton-Type Bounds for Blot-Correcting Codes}.
\newblock {\em IEEE Transactions on Information Theory}, 42(3):1021--1023,
  1996.

\bibitem{Bridwell1970BurstCorrection}
J.~D. Bridwell and J.~K. Wolf.
\newblock {Burst distance and multiple-burst correction}.
\newblock {\em The Bell System Technical Journal}, 49(5):889--909, 1970.

\bibitem{Brouwer1989Distance-RegularGraphs}
A.~E. Brouwer, A.~M. Cohen, and A.~Neumaier.
\newblock {\em {Distance-Regular Graphs}}.
\newblock Springer, 1989.

\bibitem{byrne2021fundamental}
E.~Byrne, H.~Gluesing-Luerssen, and A.~Ravagnani.
\newblock Fundamental properties of sum-rank-metric codes.
\newblock {\em IEEE Transactions on Information Theory}, 67(10):6456--6475,
  2021.

\bibitem{chowdhury2016inclusion}
A.~Chowdhury.
\newblock Inclusion matrices and the mds conjecture.
\newblock {\em The Electronic Journal of Combinatorics}, 23:\#P4.29, 2016.

\bibitem{Cvetkovic1971GraphsSpectra}
D.~M. Cvetkovi{\'{c}}.
\newblock {Graphs and their spectra}.
\newblock {\em Publikacije Elektrotehni{\v{c}}kog fakulteta. Serija Matematika
  i fizika}, 354/356:1--50, 1971.

\bibitem{Cvetkovic2009ApplicationsLiterature}
D.~M. Cvetkovi{\'{c}}.
\newblock {Applications of graph spectra: an introduction to the literature}.
\newblock {\em Zbornik Radova}, 21:7--32, 2009.

\bibitem{Cvetkovic1980SpectraApplications}
D.~M. Cvetkovi{\'{c}}, M.~Doob, and H.~Sachs.
\newblock {\em {Spectra of Graphs: Theory and Applications}}.
\newblock Academic Press, 1980.

\bibitem{Das2013LaplacianGraph}
K.~C. Das and J.~M. Guo.
\newblock {Laplacian eigenvalues of the second power of a graph}.
\newblock {\em Discrete Mathematics}, 307(313):626--634, 2013.

\bibitem{Delsarte1973AnTheory}
P.~Delsarte.
\newblock {\em {An algebraic approach to the association schemes of coding
  theory}}.
\newblock PhD thesis, 1973.

\bibitem{Delsarte1978BilinearTheory}
P.~Delsarte.
\newblock {Bilinear forms over a finite field, with applications to coding
  theory}.
\newblock {\em Journal of Combinatorial Theory, Series A}, 25(3):226--241,
  1978.

\bibitem{Delsarte1975AlternatingGFq}
P.~Delsarte and J.~M. Goethals.
\newblock {Alternating bilinear forms over GF(q)}.
\newblock {\em Journal of Combinatorial Theory, Series A}, 19(1):26--50, 1975.

\bibitem{Delsarte1981BoundsCodes}
P.~Delsarte and P.~Piret.
\newblock {Bounds and Constructions for Binary Asymmetric Error-Correcting
  Codes}.
\newblock {\em IEEE Transactions on Information Theory}, 27(1):125--128, 1981.

\bibitem{Devos2013AveragePowers}
M.~Devos, J.~McDonald, and D.~Scheide.
\newblock {Average degree in graph powers}.
\newblock {\em Journal of Graph Theory}, 72(1):7--18, 2013.

\bibitem{Dukes2020OnCodes}
P.~J. Dukes, F.~Ihringer, and N.~Lindzey.
\newblock {On the algebraic combinatorics of injections and its applications to
  injection codes}.
\newblock {\em IEEE Transactions on Information Theory}, 66(11):1--18, 2020.

\bibitem{el2007bounds}
S.~Y. El~Rouayheb, C.~N. Georghiades, E.~Soljanin, and A.~Sprintson.
\newblock Bounds on codes based on graph theory.
\newblock In {\em 2007 IEEE International Symposium on Information Theory},
  pages 1876--1879, 2007.

\bibitem{Feng2006LinearCodes}
K.~Feng, L.~Xu, and F.~J. Hickernell.
\newblock {Linear error-block codes}.
\newblock {\em Finite Fields and their Applications}, 12(4):638--652, 2006.

\bibitem{Fiol2020ANumber}
M.~A. Fiol.
\newblock {A new class of polynomials from the spectrum of a graph, and its
  application to bound the k-independence number}.
\newblock {\em Linear Algebra and its Applications}, 605:1--20, 2020.

\bibitem{Gabidulin1997CodesRotation}
E.~M. Gabidulin and M.~Bossert.
\newblock {Codes Resistant to the Phase Rotation}.
\newblock {\em Proceedings of the 4-th Symposium on Communication and
  Applications}, pages 65--84, 1997.

\bibitem{Gabidulin1997MetricsSet}
E.~M. Gabidulin and J.~Simonis.
\newblock {Metrics Generated by a Projective Point Set}.
\newblock {\em IEEE International Symposium on Information Theory -
  Proceedings}, page 248, 1997.

\bibitem{Garcia-Claro2022OnCodes}
E.~J. Garc{\'{i}}a-Claro and I.~Guti{\'{e}}rrez.
\newblock {On {G}rid {C}odes}.
\newblock {\em arXiv:2202.10005}, 2024.

\bibitem{Haemers1995InterlacingGraphs}
W.~H. Haemers.
\newblock {Interlacing eigenvalues and graphs}.
\newblock {\em Linear Algebra and its Applications}, 226-228:593--616, 1995.

\bibitem{HirschfeldKorchmaros1996}
J.~W.~P. Hirschfeld and G.~Korchm{\'a}ros.
\newblock On the embedding of an arc into a conic in a finite plane.
\newblock {\em Finite Fields and their Applications}, 2(3):274--292, 1996.

\bibitem{Huffman2003FundamentalsCodes}
W.~C. Huffman and V.~Pless.
\newblock {\em {Fundamentals of Error-Correcting Codes}}.
\newblock Cambridge University Press, 2003.

\bibitem{jiang2004asymptotic}
T.~Jiang and A.~Vardy.
\newblock Asymptotic improvement of the gilbert-varshamov bound on the size of
  binary codes.
\newblock {\em IEEE Transactions on Information Theory}, 50(8):1655--1664,
  2004.

\bibitem{Kavi2023TheMethod}
L.~C. Kavi and M.~Newman.
\newblock {The optimal bound on the 3-independence number obtainable from a
  polynomial-type method}.
\newblock {\em Discrete Mathematics}, 346:113471, 2023.

\bibitem{Klve1981ErrorChannel}
T.~Kl{\o}ve.
\newblock {Error Correcting Codes for the Asymmetric Channel}.
\newblock Technical report, 1981.

\bibitem{krivelevich2004lower}
M.~Krivelevich, S.~Litsyn, and A.~Vardy.
\newblock A lower bound on the density of sphere packings via graph theory.
\newblock {\em International Mathematics Research Notices},
  2004(43):2271--2279, 2004.

\bibitem{lovasz1975spectra}
L.~Lov{\'a}sz.
\newblock Spectra of graphs with transitive groups.
\newblock {\em Periodica Mathematica Hungarica}, 6(2):191--195, 1975.

\bibitem{Lovasz1979OnGraph}
L.~Lov{\'{a}}sz.
\newblock {On the Shannon Capacity of a Graph}.
\newblock {\em IEEE Transactions on Information Theory}, 25(1):1--7, 1979.

\bibitem{MacWilliams1977TheoryCodes}
F.~J. MacWilliams and N.~J.~A. Sloane.
\newblock {\em {The theory of error-correcting codes}}.
\newblock Elsevier, 1977.

\bibitem{martinez2019reliable}
U.~Mart{\'\i}nez-Pe{\~n}as and F.~R. Kschischang.
\newblock Reliable and secure multishot network coding using linearized
  reed-solomon codes.
\newblock {\em IEEE Transactions on Information Theory}, 65(8):4785--4803,
  2019.

\bibitem{Rao1975AsymmetricMemories}
T.~R.~N. Rao and A.~S. Chawla.
\newblock {Asymmetric error codes from some LSI semiconductor memories}.
\newblock In {\em The annual Southeastern Symposium on System Theory}, pages
  170--171, 1975.

\bibitem{SR2025}
G.~Riccardi and H.~Sauerbier~Couvée.
\newblock {Fundamental notions of projective and scale-translation-invariant
  metrics in coding theory}.
\newblock {\em arXiv:2505.06388}, 2025.

\bibitem{Roy2009BoundsSubspaces}
A.~Roy.
\newblock {Bounds for codes and designs in complex subspaces}.
\newblock {\em Journal of Algebraic Combinatorics}, 31:1--32, 2009.

\bibitem{segre1955curve}
B.~Segre.
\newblock Curve razionali normali ek-archi negli spazi finiti.
\newblock {\em Annali di Matematica Pura ed Applicata}, 39:357--379, 1955.

\bibitem{shannon1948mathematical}
C.~E. Shannon.
\newblock A mathematical theory of communication.
\newblock {\em The Bell system technical journal}, 27(3):379--423, 1948.

\bibitem{shehadeh2021space}
M.~Shehadeh and F.~R. Kschischang.
\newblock Space--time codes from sum-rank codes.
\newblock {\em IEEE Transactions on Information Theory}, 68(3):1614--1637,
  2021.

\bibitem{Sole1986TheScheme}
P.~Sole.
\newblock {The Lee association scheme}.
\newblock In {\em International Colloquium on Coding Theory and Applications},
  pages 45--55, 1986.

\bibitem{Ulrich1957NonBinaryCodes}
W.~Ulrich.
\newblock {Non‐Binary Error Correction Codes}.
\newblock {\em The Bell System Technical Journal}, 36(6):1341--1388, 1957.

\bibitem{Varshamov1964EstimateRussian}
R.~R. Varshamov.
\newblock {Estimate of the number of signals in codes with correction of
  nonsymmetric errors (in Russian)}.
\newblock {\em Automatika i Telemekhanika}, 25(11):1628--1629, 1964.

\bibitem{Varshamov1965OnCodes}
R.~R. Varshamov.
\newblock {On theory of asymmetric codes}.
\newblock {\em Doclady of AS USSR}, 164(4):757--760, 1965.

\bibitem{Villalba2016OnCodes}
L.~J.~G. Villalba, A.~L.~S. Orozco, and M.~Blaum.
\newblock {On multiple burst-correcting MDS codes}.
\newblock {\em Journal of Computational and Applied Mathematics}, 295:170--174,
  2016.

\bibitem{Voloch1991}
J.~F. Voloch.
\newblock Complete arcs in galois planes of nonsquare order.
\newblock In C.~D. Godsil and R.~K. Guy, editors, {\em Advances in finite
  geometries and designs (Chelwood Gate, 1990)}, Oxford Science Publications,
  pages 401--406. Oxford University Press, New York, 1991.

\end{thebibliography}




\end{document}